\pdfoutput=1
\documentclass[11pt]{amsart}
\usepackage[utf8]{inputenc}
\usepackage{amssymb,amsmath,amsthm,enumerate,enumitem,colonequals,mlmodern,tikz-cd,microtype,xypic}
\usepackage[cal=euler,bfcal,bb=px,bfbb]{mathalpha}
\usepackage[top=3.75cm, bottom=3cm, left=3.5cm, right=3.5cm]{geometry}
\setlist[enumerate]{font=\upshape}

\makeatletter
\@namedef{subjclassname@2020}{\textup{2020} Mathematics Subject Classification}
\makeatother

\usepackage{xcolor}
\colorlet{darkblue}{blue!55!black}
\colorlet{darkcyan}{cyan!50!black}
\colorlet{darkgreen}{green!60!black}

\usepackage{hyperref}
\hypersetup{
    colorlinks=true,
    linkcolor= darkblue,
    urlcolor= darkcyan,
    citecolor= darkgreen,
}

\def\eqref#1{\textcolor{darkblue}{(\ref{#1})}}

\PassOptionsToPackage{hyphens}{url}
\usepackage{hyperref}
\hypersetup{bookmarksdepth=2}

\usepackage[nameinlink]{cleveref} %
\Crefformat{section}{#2\S#1#3} %
\Crefmultiformat{section}{#2\S\S#1#3}{ and~#2#1#3}{, #2#1#3}{, and~#2#1#3}
\crefname{hypothesis}{hypothesis}{hypotheses}
\Crefname{hypothesis}{Hypothesis}{Hypotheses}

\usepackage[pagewise]{lineno}
\let\oldequation\equation
\let\oldendequation\endequation
\renewenvironment{equation}{\linenomathNonumbers\oldequation}{\oldendequation\endlinenomath}
\expandafter\let\expandafter\oldequationstar\csname equation*\endcsname
\expandafter\let\expandafter\oldendequationstar\csname endequation*\endcsname
\renewenvironment{equation*}{\linenomathNonumbers\oldequationstar}{\oldendequationstar\endlinenomath}
\let\oldalign\align
\let\oldendalign\endalign
\renewenvironment{align}{\linenomathNonumbers\oldalign}{\oldendalign\endlinenomath}
\expandafter\let\expandafter\oldalignstar\csname align*\endcsname
\expandafter\let\expandafter\oldendalignstar\csname endalign*\endcsname
\renewenvironment{align*}{\linenomathNonumbers\oldalignstar}{\oldendalignstar\endlinenomath}


\theoremstyle{plain}
\numberwithin{equation}{section}
\newtheorem{theorem}[equation]{Theorem}
\newtheorem{lemma}[equation]{Lemma}
\newtheorem{proposition}[equation]{Proposition}
\newtheorem{corollary}[equation]{Corollary}

\newtheorem{introthm}{Theorem}

\theoremstyle{definition}
\newtheorem{definition}[equation]{Definition}
\newtheorem{example}[equation]{Example}

\theoremstyle{remark}
\newtheorem{remark}[equation]{Remark}

\newtheorem{notation}[equation]{Notation}

\newtheorem{claim}{Claim}

\newtheorem*{ack}{Acknowledgements}

\title[Compact approximation, descent, algebraic stacks]{Compact approximation and descent \\ for algebraic stacks}

\author[J. ~Hall]{Jack Hall}
\address{J.~Hall, School of Mathematics \& Statistics\\The University of Melbourne\\Parkville,
  VIC, 3010\\Australia}
\email{jack.hall@unimelb.edu.au}

\author[A. ~Lamarche]{Alicia Lamarche}
\address{A.~Lamarche, Yau Mathematical Sciences Center, Jingzhai, Tsinghua University\\ Haidian District, Beijing, China}
\email{lamarche@mail.tsinghua.edu.cn}

\author[P.~Lank]{Pat Lank}
\address{P.~Lank,
Dipartimento di Matematica “F. Enriques”, Universit\`{a} degli Studi di Milano, Via Cesare
Saldini 50, 20133 Milano, Italy}
\email{plankmathematics@gmail.com}

\author[F. ~Peng]{Fei Peng}
\address{F.~Peng,
School of Mathematics \& Statistics\\The University of Melbourne\\Parkville,
  VIC, 3010\\Australia}
\email{pengf2@student.unimelb.edu.au}

\date{\today}



\begin{document}

\begin{abstract}
    This work focuses on approximation and generation for the derived category of complexes with quasi-coherent cohomology on algebraic stacks. Our methods establish that approximation by compact objects descends along covers that are quasi-finite and flat. This generalizes a result of Lipman--Neeman for schemes and extends a related result known for algebraic spaces. We also study the behavior of generation under the derived pushforward and pullback of a morphism between algebraic stacks.
\end{abstract}

\maketitle

\section{Introduction}
\label{sec:intro}

Let $X$ be a Noetherian scheme. A well-known result of Lipman and Neeman states that any bounded above complex of coherent sheaves on $X$ can be approximated arbitrarily well by a perfect complex \cite[Theorem\ 4.2]{Lipman/Neeman:2007}. Lipman--Neeman actually prove a result in the non-Noetherian situation: on a quasi-compact and quasi-separated scheme, pseudo-coherent complexes can be approximated arbirarily well by perfect complexes. Our first main result is the following generalization.
\begin{introthm}\label{introthm:approximation_for_quasifinite}
    Let $\mathcal{X}$ be a quasi-compact and quasi-separated algebraic stack. If either
    \begin{enumerate}
        \item $\mathcal{X}$ has quasi-finite and separated diagonal, or
        \item $\mathcal{X}$ is Deligne--Mumford and of characteristics 0,
    \end{enumerate}
    then approximation by compact complexes holds for $\mathcal{X}$.
\end{introthm}
See \Cref{def:approximation} for a precise definition of approximation by compact complexes. Since every compact complex on a quasi-compact and quasi-separated stack is perfect \cite[Lemma 4.4]{Hall/Rydh:2017},  \Cref{introthm:approximation_for_quasifinite} generalizes the aforementioned result of Lipman and Neeman for schemes and \cite[\href{https://stacks.math.columbia.edu/tag/08H}{Tag 08HP}]{StacksProject} for algebraic spaces.  %
Our second main result is a variant of \Cref{introthm:approximation_for_quasifinite} for stacks with infinite stabilizers.

\begin{introthm}\label{introthm:approximation_for_good_moduli_spaces}
  Let $\mathcal{X}$ be a quasi-compact and quasi-separated algebraic
  stack that admits a good moduli space. If $\mathcal{X}$ has affine
  stabilizers and separated diagonal, then approximation by compact
  complexes holds for $\mathcal{X}$.
\end{introthm}

Both \Cref{introthm:approximation_for_quasifinite} and \ref{introthm:approximation_for_good_moduli_spaces} are consequences of a general descent result for approximation by compact complexes (\Cref{thm:etale_descent_approximation}). This confirms an expectation of the first author and Rydh in \cite{Hall/Rydh:2017}. Our argument for \Cref{thm:etale_descent_approximation} is inspired by the proof of \cite[Theorem\ C]{Hall/Rydh:2017}, which we follow closely. The key ingredient is quasi-finite and flat d\'evissage developed by the first author and Rydh in \cite{Hall/Rydh:2018}. 

\Cref{introthm:approximation_for_quasifinite} and \ref{introthm:approximation_for_good_moduli_spaces} have some interesting applications. For instance, \Cref{introthm:approximation_for_quasifinite} was used to establish a variant of \cite[Theorem\ 0.1]{Neeman:2022} relating regularity and the existence of bounded t-structures on the category of perfect complexes for algebraic stacks \cite{DeDeyn/Lank/KabeerManaliRahul/Peng:2025}. In this article, we will extend various results about generation from the scheme case established in \cite{Lank:2024,Lank/Olander:2024,Dey/Lank:2024} to the case of algebraic stacks.

Recall that a morphism of finite type $\pi \colon \mathcal{Y}\to \mathcal{S}$ between Noetherian algebraic stacks is said to be \textbf{cohomologically proper} if $\mathbb{R}\pi_\ast E$ belongs to $D^{+}_{\operatorname{coh}}(\mathcal{S})$ for each $E$ in $\operatorname{coh}(\mathcal{Y})$. Additionally, we say $\pi \colon \mathcal{Y}\to \mathcal{S}$ is \textbf{universally cohomologically proper} if the projection $\mathcal{Y}\times_{\mathcal{S}}\mathcal{Z} \to \mathcal{Z}$ is cohomologically proper for every morphism $\mathcal{Z} \to \mathcal{S}$ \cite[$\S 2$]{Alper/Hall/BenjaminLim:2023}. For example, a proper morphism of Noetherian algebraic stacks is universally cohomologically proper \cite[Example 2.6]{Alper/Hall/BenjaminLim:2023} and universally cohomologically proper morphisms are always universally closed \cite[Proposition 2.4.5]{Halpern-Leistner/Preygel:2023}. As a consequence of our main theorems, we obtain the following amplification of \cite[Lemma\ 2.17]{Ballard/Favero:2012}.

\begin{introthm}\label{introthm:rouq_dim_bound_tame_stack_coarse_moduli}
    Let $f\colon \mathcal{X} \to Y$ be a universally cohomologically proper, surjective and concentrated morphism of algebraic stacks, where $Y$ is a Noetherian algebraic space that is irreducible, Jacobson and catenary of finite Krull dimension with open regular locus. If the natural map $\mathcal{O}_Y \to \mathbb{R}f_\ast \mathcal{O}_{\mathcal{X}}$ splits, then $\dim Y \leq \dim D^b_{\operatorname{coh}}(\mathcal{X})$.
\end{introthm}

Note that $\dim D^b_{\operatorname{coh}}(\mathcal{X})$ refers to the Rouquier dimension $D^b_{\operatorname{coh}}(\mathcal{X})$ as a triangulated category \cite{Rouquier:2008}. Our argument for \Cref{introthm:rouq_dim_bound_tame_stack_coarse_moduli} relies on the following descendability (see \Cref{def:descendable_object}) result for coverings of concentrated algebraic stacks. 

\begin{introthm}\label{introthm:descendable_submersive}
    Let $\mathcal{S}$ be an algebraic stack that is concentrated and Noetherian. If $f \colon \mathcal{Y} \to \mathcal{S}$ is a universally submersive morphism of algebraic stacks of finite type, then $f$ is descendable.
\end{introthm}

\Cref{introthm:descendable_submersive} extends previous results of Bhatt and Scholze \cite[Theorem\ 11.25]{Bhatt/Scholze:2017} for schemes, the first author \cite[Theorem\ 7.1]{Hall:2022} for faithfully flat coverings of finite presentation, and the first author and Priver \cite[Theorem\ B.1]{Hall/Priver:2024} for finite morphisms to tame stacks. A pleasant consequence of this circle of ideas is the following base change result for generation. Note that this result is sharp (\Cref{ex:bad_behavior_1} and \Cref{ex:bad_behavior_2}).

\begin{introthm}\label{introthm:base_change_generation}
    Let $k \subseteq \ell$ be a field extension, where $k$ is perfect. Let $\mathcal{Y}$ be an algebraic stack of finite type over $k$ with quasi-finite and separated diagonal. Let $w \colon \mathcal{Y} \otimes_k \ell \to \mathcal{Y}$ be the induced morphism. If $G$ is a classical generator for $D^b_{\operatorname{coh}}(\mathcal{Y})$, then $w^*G$ is a classical generator for $D^b_{\operatorname{coh}}(\mathcal{Y} \otimes_k \ell)$. Moreover, if $\langle w^*G \rangle_n = D^b_{\operatorname{coh}}(\mathcal{Y} \otimes_k \ell)$ for some $n$, then $\langle G \rangle_n = D^b_{\operatorname{coh}}(\mathcal{Y})$. 
\end{introthm}

This article is organized as follows. In \Cref{sec:prelim}, we recall some basic facts about generation of triangulated categories, especially in the case of derived categories of sheaves on algebraic stacks. In \Cref{sec:approximation}, we introduce the notion of approximation by compact complexes and record some basic properties. In \Cref{sec:resolution-property}, we recall the compact resolution property and show that approximation by compact complexes holds for algebraic stacks with the compact resolution property.  In \Cref{sec:approximation_by_perfect}, we show that approximation by compact complex is local for the quasi-finite and flat topology, and use it to deduce \Cref{introthm:approximation_for_quasifinite} and \Cref{introthm:approximation_for_good_moduli_spaces}. In \Cref{sec:devissage}, we prove \Cref{introthm:descendable_submersive} and then use it and the main theorems to deduce \Cref{introthm:rouq_dim_bound_tame_stack_coarse_moduli}. We also study the behaviour of strong generation along base change and prove \Cref{introthm:base_change_generation}. In two appendices, we record a variant of Deligne's formula for local cohomology and the Verdier localization sequence for the bounded derived category of $\mathcal{O}_{\mathcal{X}}$-modules with coherent cohomology, which are used in \Cref{sec:resolution-property}.

\begin{remark}
  Another result contained in \cite{Lipman/Neeman:2007}, which we do
  not establish, is the existence of a uniform bound for maps from a
  compact generator to a general object in the unbounded derived
  category. While it is expected that this result extends to
  quasi-compact algebraic stacks with quasi-finite and separate
  diagonal, it appears to require a deeper understanding of
  Koszul-like complexes to run a d\'evissage effectively.
\end{remark}

\subsection{Conventions}
\label{sec:intro_convention}

We will use symbols $X,Y,\ldots$ to denote schemes and $\mathcal{X},\mathcal{Y},\ldots$ to denote algebraic stacks. Our work will follow the conventions and terminology set forth in \cite{StacksProject,Hall/Rydh:2017}. Let $\mathcal{X}$ be an algebraic stack. The following categories will be of interest throughout this document. See \cite[$\S 1.3$]{Hall/Rydh:2017} for details on functors between these categories induced from a morphism of algebraic stacks.

\begin{itemize}
    \item $\operatorname{Mod}(\mathcal{X})$ denotes the Grothendieck abelian category of $\mathcal{O}_\mathcal{X}$-modules on the lisse-\'{e}tale topos of $\mathcal{X}$.
    \item $\operatorname{Qcoh}(\mathcal{X})$ denotes the Grothendieck abelian category of quasi-coherent $\mathcal{O}_\mathcal{X}$-modules on the lisse-\'{e}tale topos of $\mathcal{X}$.
    \item $\operatorname{coh}(\mathcal{X})$ denotes the abelian category of coherent $\mathcal{O}_\mathcal{X}$-modules on the lisse-\'{e}tale topos of $\mathcal{X}$.
    \item If $\mathcal{A}$ is a Grothendieck abelian category, then we let $D(\mathcal{A})$ denote the unbounded derived category of $\mathcal{A}$; superscripts $+$, $-$, $\geq n$, etc., are to be interepreted as usual. Let $\phi \colon M \to N$ be a morphism in $D(\mathcal{A})$ and form the distinguished triangle $M \xrightarrow{\phi} N \to Q \to M[1]$. If $\tau^{\geq m}Q \simeq 0$, then we say that $\phi$ is an $m$-isomorphism; equivalently, $H^i(M) \to H^i(N)$ is surjective when $i=m$ and an isomorphism if $i>m$.
    \item By abuse of notation, $D(\mathcal{X}):= D(\operatorname{Mod}(\mathcal{X}))$ denotes the unbounded derived category of $\operatorname{Mod}(\mathcal{X})$.
    \item $D_{\operatorname{qc}}(\mathcal{X})$ is the strictly full subcategory of $D(\mathcal{X})$ with quasi-coherent cohomology groups.
    \item $D^b_{\operatorname{coh}}(\mathcal{X})$ is the strictly full subcategory of the objects in $D_{\operatorname{qc}}(\mathcal{X})$ with bounded and coherent cohomology.
    \item $D^{+}_{\operatorname{coh}}(\mathcal{X})$ (resp., $D^{-}_{\operatorname{coh}}(\mathcal{X})$) is the strictly full subcategory of the objects in $D_{\operatorname{qc}}^{+} (\mathcal{X})$ (resp. $D_{\operatorname{qc}}^{-} (\mathcal{X})$) with coherent cohomology.
    \item $\operatorname{Perf}(\mathcal{X})$ will denote the triangulated subcategory of $D_{\operatorname{qc}}(\mathcal{X})$ consisting of perfect complexes.
    \item $D_{\operatorname{qc},Z}(\mathcal{X})$ will denote those
      objects in $D_{\operatorname{qc}}(\mathcal{X})$ whose
      (cohomological) support is contained in $Z$.
\end{itemize}

\begin{ack}
    Jack Hall and Fei Peng were partially supported by the Australian Research Council DP210103397 and FT210100405. Pat Lank was supported under the ERC Advanced Grant 101095900-TriCatApp and partially supported by the National Science Foundation under Grant No. DMS-1928930 while in residence at the Simons Laufer Mathematical Sciences Institute (formerly MSRI). The authors are grateful to Timothy De Deyn, Kabeer Manali Rahul, Amnon Neeman, and David Rydh for looking at earlier versions of our work. 
\end{ack}

\section{Preliminaries}
\label{sec:prelim}

This section briefly discusses notions of generation in triangulated categories and various derived categories associated to algebraic stacks. We will freely use the content found in \cite{BVdB:2003, Rouquier:2008, Neeman:2021, ABIM:2010,Hall/Rydh:2017}. Throughout this section, $\mathcal{T}$ will be a triangulated category with shift functor $[1]\colon \mathcal{T} \to \mathcal{T}$ and $\mathcal{T}^c \subseteq \mathcal{T}$ will be its full triangulated subcategory of compact objects.

\subsection{Small generation}
\label{sec:prelim_gen}
Let
$\mathcal{S},\mathcal{S}^\prime$ be full subcategories of
$\mathcal{T}$. Set $\mathcal{S} \star \mathcal{S}^\prime$ to be the
smallest strictly (i.e., closed under isomorphs) full subcategory of
$\mathcal{S}$ consisting of all objects $E$ in $\mathcal{T}$ that can
be obtained as an extension of objects $S$ is in $\mathcal{S}$ and
$S^\prime$ is in $\mathcal{S}^\prime$:
\begin{displaymath}
  S \to E \to S^\prime \to S[1].
\end{displaymath}
We say that $\mathcal{S}$ is \textbf{thick} if it is closed under retracts and is a triangulated subcategory of $\mathcal{T}$. Let  $\langle \mathcal{S} \rangle$ denote the smallest thick subcategory containing $\mathcal{S}$ in $\mathcal{T}$. Also, $\operatorname{add}(\mathcal{S})$ will denote the smallest strictly full subcategory of $\mathcal{T}$ containing $\mathcal{S}$ that is closed under shifts, finite coproducts, and retracts. We now set  $\langle \mathcal{S} \rangle_0 := \operatorname{add}(0)$,  $\langle \mathcal{S} \rangle_1 := \operatorname{add}(\mathcal{S})$  and inductively define
\[
  \langle \mathcal{S} \rangle_n := \operatorname{add} \{
  \operatorname{cone}(\phi) : \phi \in
  \operatorname{Hom}_{\mathcal{T}} (\langle \mathcal{S} \rangle_{n-1},
  \langle \mathcal{S} \rangle_1) \}.
\]
It follows that 
\begin{equation}
  \langle \mathcal{S} \rangle_0 \subseteq \langle \mathcal{S} \rangle_1 \subseteq \cdots \subseteq \bigcup^{\infty}_{n=0} \langle \mathcal{S} \rangle_n = \langle \mathcal{S} \rangle. \label{eq:thick_subcategory_filtration}
\end{equation}
Now let $E,G$ be objects of $\mathcal{T}$. We say that $E$ is
\textbf{finitely built} by $\mathcal{S}$ if $E$ belongs to
$\langle \mathcal{S} \rangle$. The \textbf{level} of $E$ with respect
to $\mathcal{S}$, denoted
$\operatorname{level}_\mathcal{T}^{\mathcal{S}}(E)$, is the minimal
non-negative integer $n$ such that $E$ is in
$\langle \mathcal{S} \rangle_n$; it will be set to $\infty$ if no such
integer exists. Also, $\mathcal{S}$ is said to \textbf{classically
  generate} $\mathcal{T}$ if
$\langle \mathcal{S} \rangle = \mathcal{T}$. If $\mathcal{S}$ consists
of a single object $G$ and classically generates, then we will call
$G$ a \textbf{classical generator}.

We say that  $\mathcal{S}$ \textbf{strongly generates} $\mathcal{T}$ if for some $n\geq 0$ one has $\langle \mathcal{S} \rangle_{n+1} = \mathcal{T}$, and the minimal $n$ such that $\operatorname{level}_\mathcal{T}^{\mathcal{S}} (E)\leq n+1$ for all $E$ in $ \mathcal{T}$ is called the \textbf{generation time} of $\mathcal{S}$. If $\mathcal{S}$ consists of a single object $G$ and strongly generates, then we will say that $G$ is a \textbf{strong generator}.
Let $\dim \mathcal{T}$ denote the \textbf{Rouquier dimension} of $\mathcal{T}$, which is the smallest integer $d$ such that $\langle G \rangle_{d+1} = \mathcal{T}$ for some object $G$ in $ \mathcal{T}$.

An incomplete list of generation results for schemes includes
\cite{Olander:2023,Rouquier:2008,BILMP:2023,Iyengar/Takahashi:2019,Dey/Lank/Takahashi:2023,Aoki:2021,Dey/Lank:2024a,Elagin/Lunts/Schnurer:2020}. We
refer the reader to \cite[Example 2.4]{DeDeyn/Lank/ManaliRahul:2024}
for an explicit list of such objects. We will defer some examples that
are relevant to this article until the end of this section.

\subsection{Big generation}
\label{sec:prelim_big_gen}
We now assume that $\mathcal{T}$ admits all small coproducts. Consider
the following additive subcategories of $\mathcal{T}$:
$\operatorname{Add}(\mathcal{S})$ will be the smallest strictly full
subcategory containing $\mathcal{S}$ and closed under retracts,
shifts, and small coproducts. Next, we inductively define
$\overline{\langle \mathcal{S} \rangle_0} := \operatorname{Add}(0)$,
$\overline{\langle \mathcal{S} \rangle_1} :=
\operatorname{Add}(\mathcal{S})$ and
 \[
   \overline{\langle \mathcal{S} \rangle_n}:= \operatorname{Add} \{ \operatorname{cone}(\phi) : \phi \in \operatorname{Hom}_{\mathcal{T}} (\overline{\langle \mathcal{S} \rangle_{n-1}}, \overline{\langle \mathcal{S} \rangle_1}) \}.
 \]

\subsection{Boundedness and generation for algebraic stacks}
\label{sec:prelim_boundedness}

Let
$\pi \colon \mathcal{Y} \to \mathcal{X}$ be a quasi-compact and
quasi-separated morphism of algebraic stacks. Let $n\geq 0$; then
$\pi$ is said to have \textbf{cohomological dimension $\leq n$} if
$\mathbb{R}^j \pi_\ast E =0$ for all $j> n$ and $E$ in
$\operatorname{Qcoh}(\mathcal{Y})$. We say that $\pi$ has
\textbf{finite cohomological dimension} if it has cohomological
dimension $\leq n$ for some $n$. We say that $\pi$ is
\textbf{concentrated} if for every morphism
$\mathcal{Z} \to \mathcal{X}$ with $\mathcal{Z}$ a quasi-compact
quasi-separated algebraic stack, the projection morphism
$\mathcal{Y} \times_{\mathcal{X}} \mathcal{Z} \to \mathcal{Z}$ has
finite cohomological dimension. If $\mathcal{X}$ is quasi-compact and
quasi-separated, then we say it is \textbf{concentrated} if the
structure morphism $\mathcal{X} \to \operatorname{Spec}( \mathbf{Z})$ is
so.

\begin{example}\label{ex:concentrated}
  In general, $\pi$ being concentrated is subtle, but it does
  happen frequently \cite[\S2]{Hall/Rydh:2017}: representable by algebraic spaces, tame, or relatively affine stabilizers in characteristic $0$. Also, if $\mathcal{X}$ has
  quasi-affine diagonal, then finite cohomological dimension and
  concentrated are equivalent.
\end{example}
We also record here the following simple lemma.
\begin{lemma}\label{lem:concentrated_composition_inner_most}
    Let $f\colon \mathcal{Z} \to \mathcal{Y}$ and $g\colon \mathcal{Y} \to \mathcal{X}$ be quasi-compact and quasi-separated morphisms of algebraic stacks. If $f\circ g$ is concentrated, then $f$ is concentrated.
\end{lemma}
\begin{proof}
This follows from a standard argument using the graph and \cite[Lem.~2.5]{Hall/Rydh:2017}.
\end{proof}

Recall that a morphism of finite type $\pi \colon \mathcal{Y}\to \mathcal{S}$ between Noetherian algebraic stacks is said to be \textbf{cohomologically proper} if $\mathbb{R}\pi_\ast E$ belongs to $D^{+}_{\operatorname{coh}}(\mathcal{S})$ for each $E$ in $\operatorname{coh}(\mathcal{Y})$. Additionally, we say $\pi \colon \mathcal{Y}\to \mathcal{S}$ is \textbf{universally cohomologically proper} if the projection $\mathcal{Y}\times_{\mathcal{S}}\mathcal{Z} \to \mathcal{Z}$ is cohomologically proper for every morphism $\mathcal{Z} \to \mathcal{S}$. See \cite[$\S 2$]{Alper/Hall/BenjaminLim:2023} for details. For example, a proper morphism of Noetherian algebraic stacks is universally cohomologically proper \cite[Example 2.6]{Alper/Hall/BenjaminLim:2023}. Note that universally cohomologically proper morphisms are universally closed \cite[Proposition 2.4.5]{Halpern-Leistner/Preygel:2023}.

Let $\beta$ be a cardinal. If $\mathcal{X}$ is quasi-compact and
quasi-separated, then we say that it satisfies the
\textbf{$\beta$-Thomason condition} if
$D_{\operatorname{qc}}(\mathcal{X})$ is compactly generated by a
collection of cardinalilty at most $\beta$ and for each closed subset
$Z$ of $|\mathcal{X}|$ with quasi-compact complement, there is a
perfect complex $P$ on $\mathcal{X}$ such that
$\operatorname{supph}(P)=Z$. We say that $\mathcal{X}$ is
\textbf{$\beta$-crisp} if for every representable by algebraic spaces \'{e}tale separated
quasi-compact morphism $\mathcal{W} \to \mathcal{X}$, $\mathcal{W}$
satisfies the $\beta$-Thomason condition.  We will often omit the
$\beta$ to instead say $\mathcal{X}$ satisfies the `Thomason
condition' and is `crisp'.

\begin{remark}
  An equivalent characterization for an algebraic stack to be
  $\beta$-crisp is that for every representable \'{e}tale
  quasi-compact separated morphism $\mathcal{W} \to \mathcal{X}$ and
  every quasi-compact open immersion $j\colon \mathcal{U} \to \mathcal{W}$, the
  triangulated category
  $D_{\operatorname{qc},|\mathcal{W}|\setminus |\mathcal{U}|}
  (\mathcal{W})$ is compactly generated by a collection of
  cardinalilty at most $\beta$ \cite[Lemmas 4.9 \&
  8.2]{Hall/Rydh:2017}.
\end{remark}

\begin{example}\label{ex:strong_generator_examples}
  Let $\mathcal{X}$ be a quasi-compact and quasi-separated algebraic
  stack. Then
  $D_{\operatorname{qc}}(\mathcal{X})^c \subseteq
  \operatorname{Perf}(\mathcal{X})$, with equality if and only if
  $\mathcal{X}$ is concentrated \cite[Remark 4.6]{Hall/Rydh:2017}. If
  $\mathcal{X}$ is regular, then
  $D^b_{\operatorname{coh}}(\mathcal{X})=
  \operatorname{Perf}(\mathcal{X})$ \cite[Proposition
  A.2]{Bergh/Lunts/Schnurer:2016}. If $\mathcal{X}$ has quasi-finite
  and separated diagonal (e.g., an algebraic space) \cite[Theorem
  A]{Hall/Rydh:2017} or is Deligne--Mumford over $\mathbf{Q}$
  \cite[Theorem 7.4]{Hall/Rydh:2018}, then
  $D_{\operatorname{qc}}(\mathcal{X})$ is $1$-crisp, and so
  $\operatorname{Perf}(\mathcal{X})$ admits a classical generator. If
  $\mathcal{X}$ is Noetherian, separated, quasiexcellent and tame of
  finite Krull dimension, then $D^b_{\operatorname{coh}}(\mathcal{X})$
  has a strong generator \cite[Theorem B.1]{Hall/Priver:2024}.
\end{example}

\section{Approximation by compact complexes}\label{sec:approximation}

This section introduces our notion of approximation for algebraic stacks. 

\begin{definition}\label{def:approximation}
    Let $\mathcal{X}$ be an algebraic stack. Let $(T, E, m)$ be a triple consisting of
        \begin{enumerate}
            \item a closed subset $T\subset|\mathcal{X}|$,
            \item a complex $E\in D_{\operatorname{qc}}(\mathcal{X})$, and
            \item an integer $m$.
        \end{enumerate}
        We say \textbf{approximation by compact complexes holds for
        $(T, E, m)$} if there exist a compact complex $P$ on
        $\mathcal{X}$, which is supported on $T$, and an $m$-isomorphism
        $P\to E$ (i.e., the map $H^i(P) \to H^i(E)$ is surjective when $i=m$ and an isomorphism if $i>m$). We say \textbf{approximation by compact complexes holds
        for $\mathcal{X}$} if for every closed subset
        $T\subset |\mathcal{X}|$ with quasi-compact complement, there
        exists an integer $r$ such that approximation by compact complexes
        holds for any triple $(T, E, m)$ as above where
        \begin{enumerate}[label=(\alph*)]
            \item $E$ is an $(m-r)$-pseudocoherent complex in $D_{\operatorname{qc}}(\mathcal{X})$, and
            \item $H^{i}(E)$ is supported on $T$ if $i\geq m-r$.
        \end{enumerate}
\end{definition}
One of our main interests in approximation by compact complexes on stacks is because of the following lemma. 
\begin{lemma}\label{lem:rouq_dim_stacks_via_approx}
    Let $\mathcal{X}$ be a Noetherian algebraic stack satisfying approximation by compacts. If $G$ is an object of $D^b_{\operatorname{coh}}(\mathcal{X})$ and $n\geq 0$, then the following are equivalent:
    \begin{enumerate}
        \item $\operatorname{Perf}(\mathcal{X})$ is contained in $\langle G \rangle_n$
        \item $\langle G \rangle_n = D^b_{\operatorname{coh}}(\mathcal{X})$.
    \end{enumerate}
\end{lemma}
\begin{proof}
    This follows from \Cref{def:approximation} and \cite[Proposition 3.5]{Lank/Olander:2024}. 
\end{proof}
We also collect here some useful refinements of the results of
\cite[Appendix A]{Hall:2022} on approximations of certain
non-pseudocoherent complexes. 

\begin{lemma}\label{lem:lift-compact}
  Let $\mathcal{X}$ be a quasi-compact and quasi-separated algebraic
  stack that satisfies approximation by compact complexes. Let
  $\phi \colon P \to M$ be a morphism in
  $D_{\operatorname{qc}}(\mathcal{X})$, where $P$ is compact and $M$
  is pseudocoherent. Then there exists an $r_P$, independent of
  $M$, such that for all $m\leq r_P$ there is an $m$-isomorphism
  $q \colon Q \to M$, where $Q$ is compact, and a lift
  $\phi' \colon P \to Q$ such that $q\circ \phi' = \phi$.
\end{lemma}

\begin{proof}
  By \cite[Lemma 4.5]{Hall/Rydh:2017}, there is an $r_P$ such that
  \[
    \operatorname{Hom}_{\mathcal{O}_{\mathcal{X}}}(P,N) \simeq
    \operatorname{Hom}_{\mathcal{O}_{\mathcal{X}}}(P,\tau^{\geq m}N),
  \]
  whenever $m\leq r_P$ for all
  $N\in D_{\operatorname{qc}}(\mathcal{X})$. If $m\leq r_P$, then
  there is a compact complex $Q$ and an $m$-isomorphism
  $q \colon Q \to M$. It follows that $\phi$ factors
  through $q$ as claimed.
\end{proof}

If $\mathcal{W}$
is a Noetherian algebraic stack, let
$\overline{\operatorname{coh}}(\mathcal{W}) \subset
\operatorname{Qcoh}(\mathcal{W})$ denote the full subcategory with
objects those quasi-coherent $\mathcal{O}_{\mathcal{W}}$-modules $M$
such that $\Gamma(\operatorname{Spec}(R),M)$ is a countably generated
$R$-module for all smooth morphisms
$\operatorname{Spec}(R) \to \mathcal{W}$. By
\cite[Lem.~A.1]{Hall:2022},
$\overline{\operatorname{coh}}(\mathcal{W})$ is a Serre subcategory of
$\operatorname{Qcoh}(\mathcal{W})$. In particular, the full
subcategory
$D^b_{\overline{\operatorname{coh}}}(\mathcal{W}) \subseteq
D_{\operatorname{qc}}(\mathcal{W})$ of bounded complexes of
$\mathcal{O}_{\mathcal{W}}$-modules with cohomology groups in
$\overline{\operatorname{coh}}(\mathcal{W})$ is a triangulated
subcategory.

\begin{proposition}\label{prop:approximation-countable}
  Let $\mathcal{X}$ be a Noetherian algebraic stack that satisfies
  approximation by compact complexes. If
  $M \in D^b_{\overline{\operatorname{coh}}}(\mathcal{X})$, then
  $M\simeq \operatorname{hocolim}_r P_r$, where
  $P_r \in D_{\operatorname{qc}}(\mathcal{X})^c$. Moreover, for each
  $s \in \mathbf{Z}$, there exists $r_s$ such that
  $\mathcal{H}^s(P_r) \to \mathcal{H}^s(M)$ is a monomorphism for all
  $r>r_s$.
\end{proposition}

\begin{proof}
  By \cite[Lemma A.3]{Hall:2022}, we may write
  $M = \operatorname{hocolim}_n M_n$, where
  $M_n \in D^b_{\operatorname{coh}}(\mathcal{X})$. Now choose a
  compact $P_1$ and a $0$-isomorphism $p_1 \colon P_1 \to M_1$. Let
  $m_1 = 0$. By \Cref{lem:lift-compact}, we may inductively
  factor the composition $P_i \to M_i \to M_{i+1}$ to an
  $m_{i+1}=\min\{r_{P_i},-i-1\}$-approximation
  $p_{i+1}\colon P_{i+1} \to M_{i+1}$. Let $P = \operatorname{hocolim}_i P_i$;
  then there is an induced morphism $P \to M$, which is clearly a quasi-isomorphism.
\end{proof}

\section{Algebraic stacks with the (compact) resolution property}
\label{sec:resolution-property}

In this section, we establish that approximation by compact complexes
holds for algebraic stacks with the compact resolution property. We
begin with some definitions. 

\begin{definition}(see \cite[$\S 7$]{Hall/Rydh:2017})\label{def:compact_RP}
    Let $\mathcal{X}$ be an algebraic stack. A vector bundle $E$ on
    $\mathcal{X}$ is \textbf{compact} if $E[0]$ is a compact object of
    $D_{\operatorname{qc}}(\mathcal{X})$.  If $\mathcal{X}$ is
    quasi-compact and quasi-separated, we say that it has the
    \textbf{(compact) resolution property} if every quasi-coherent sheaf
    on $\mathcal{X}$ admits an epimorphism from a direct sum of (compact)
    vector bundles.
\end{definition}

A quasi-compact and quasi-separated algebraic stack with affine
stabilizers has the resolution property if, and only if, it can be
expressed as a quotient of a quasi-affine scheme by
$\operatorname{GL}_N$ for some $N>0$
\cite{Totaro:2004,Gross:2017}. For example, every quasi-projective
scheme has the (compact) resolution property. If $\mathcal{X}$ is
concentrated, then every vector bundle is compact; hence, the resolution
property coincides with the compact resolution property. The following
notation will be useful.

\begin{notation}
    \hfill
    \begin{enumerate}
        \item Let $\mathcal{X}$ be an algebraic stack. If $N$ belongs to $D_{\operatorname{qc}}(\mathcal{X})$, then set 
        \begin{displaymath}
            \operatorname{pd}(N) := \sup_{M \in \operatorname{QCoh}(\mathcal{X})}\{ d \colon \operatorname{Ext}^d_{\mathcal{O}_{\mathcal{X}}}(N,M) \neq 0\}.
        \end{displaymath}
        We will also set
        $\operatorname{cd}(\mathcal{X}):=\operatorname{pd}(\mathcal{O}_{\mathcal{X}})$, the \textbf{cohomological dimension} of $\mathcal{X}$ \cite[Proposition 4.5 \& Remark 4.6]{Hall/Rydh:2017}. If $N$ is compact,
        then $\operatorname{pd}(N) < \infty$ \cite[Lemma
        4.5]{Hall/Rydh:2017}.
      \item If $h \colon \mathcal{V} \to \mathcal{X}$ is a
        quasi-compact and quasi-separated morphism of algebraic
        stacks, then we set
        \begin{displaymath}
          \operatorname{cd}(h) := \sup_{\operatorname{Spec}(A) \to \mathcal{X}} \operatorname{cd}(\mathcal{V} \times_{\mathcal{X}} \operatorname{Spec}(A)).
        \end{displaymath}
        If $\mathcal{X}$ has quasi-affine diagonal, then 
        $\operatorname{cd}(h) \leq \operatorname{cd}(\mathcal{V}
        \times_{\mathcal{X}} \operatorname{Spec} (B))$ for every faithfully flat covering
        $\operatorname{Spec} (B) \to \mathcal{X}$ \cite[Lemma 2.2(6)]{Hall/Rydh:2017}.
    \end{enumerate}
\end{notation}

The following lemma will be helpful for us and is similar to \cite[Lemma 7.6]{Hall/Rydh:2017}.

\begin{lemma}\label{lem:bounded}
    Let $\mathcal{X}$ be an algebraic stack. Suppose $E$ is a vector bundle on $\mathcal{X}$.
    \begin{enumerate}
    \item  \label{LI:bounded:tensor} If $\operatorname{pd}(E) \leq r$ and $V$ is a vector bundle on $\mathcal{X}$, then $\operatorname{pd}(E\otimes_{\mathcal{O}_{\mathcal{X}}} V) \leq r$.
    \item  \label{LI:bounded:dual} If $\operatorname{pd}(E) \leq r$, then $\operatorname{pd}(E^\vee) \leq r$.
    \item \label{LI:bounded:2of3} Consider a short exact sequence of vector bundles on $\mathcal{X}$:
    \begin{displaymath}
        0 \to E \to F \to G \to 0.
      \end{displaymath}
      If $\operatorname{pd}(F) \leq r$ and one of
      $\operatorname{pd}(E)$ or $\operatorname{pd}(F)$ are $\leq r$,
      then the other is too.
    \end{enumerate}
\end{lemma}

\begin{proof}
    \hfill
    \begin{enumerate}
        \item If $V$ is a vector bundle, then there exists an isomorphism: 
        \begin{displaymath}
            \operatorname{Ext}^d_{\mathcal{O}_{\mathcal{X}}}(E\otimes_{\mathcal{O}_{\mathcal{X}}} V, M) \cong \operatorname{Ext}^d_{\mathcal{O}_{\mathcal{X}}}(E, V^\vee \otimes_{\mathcal{O}_{\mathcal{X}}} M).
        \end{displaymath}
        The desired claim now follows directly.
        \item We know that the map $E^\vee \to E^\vee \otimes_{\mathcal{O}_{\mathcal{X}}} E \otimes_{\mathcal{O}_{\mathcal{X}}} E^\vee$ admits a retraction, so it suffices to prove that $\operatorname{pd}(E^\vee \otimes_{\mathcal{O}_{\mathcal{X}}} E \otimes_{\mathcal{O}_{\mathcal{X}}} E^\vee) \leq r$, which follows from \eqref{LI:bounded:tensor}. \item If both $\operatorname{pd}(F)$, $\operatorname{pd}(G)$ are at most $r$, then it follows from the long exact sequence in $\operatorname{Ext}$ that $\operatorname{pd}(E) \leq r$. Consider the case where $\operatorname{pd}(E)$, $\operatorname{pd}(F) \leq r$. We can dualize the defining sequence to obtain a short exact sequence of vector bundles: 
        \begin{displaymath}
            0 \to G^\vee \to F^\vee \to E^\vee \to 0.
        \end{displaymath}
        By \eqref{LI:bounded:dual}, $\operatorname{pd}(F^\vee)$ and $\operatorname{pd}(E^\vee)$ are both $\leq r$. Now from the case already considered, we see that $\operatorname{pd}(G^\vee) \leq r$. The desired claim follows if we apply \eqref{LI:bounded:dual} to $G^\vee$.\qedhere
    \end{enumerate}
\end{proof}

\begin{lemma}\label{lem:bound-strata}
  Let $j \colon \mathcal{U} \to \mathcal{X}$ be an open immersion of
  quasi-compact and quasi-separated algebraic stacks. Let
  $i\colon \mathcal{Z} \hookrightarrow \mathcal{X}$ be a finitely
  presented closed immersion with
  $|\mathcal{X}|\setminus |i(\mathcal{Z})| = |\mathcal{U}|$. If $E$ is
  a compact vector bundle on $\mathcal{X}$, then
    \begin{displaymath}
        \operatorname{pd}(E) \leq \max\{ \operatorname{pd}(j^\ast E), \operatorname{pd}(i^\ast E) + \operatorname{cd}(j) + 1\}.
    \end{displaymath}
\end{lemma}

\begin{proof}
  Argue as in \cite[Lemma 2.7]{Hall/Rydh:2015}.
\end{proof}
The following result is surprisingly difficult to establish in positive and mixed characteristics. Perhaps there is a simpler approach, but we were unable to find one.
\begin{theorem}\label{thm:uniform-bound-ext-dim}
    Let $\mathcal{X}$ be a quasi-compact and quasi-separated algebraic stack. If $\mathcal{X}$ has the compact resolution property, then there exists an integer $d$ such that $\operatorname{pd}(E) \leq d$ for all compact vector bundles $E$ on $\mathcal{X}$. Moreover, if $E$ is a compact vector bundle on $\mathcal{X}$, $M$ is in $D_{\operatorname{qc}}(\mathcal{X})$ and $s$ is an integer, then one has the following isomorphism:
    \begin{displaymath}
        \tau^{\geq s}\mathbb{R}\operatorname{Hom}_{\mathcal{O}_{\mathcal{X}}}(E,M) \cong \tau^{\geq s}\mathbb{R}\operatorname{Hom}_{\mathcal{O}_{\mathcal{X}}}(E,\tau^{\geq s-d}M).
    \end{displaymath}
\end{theorem}

\begin{proof}
    There is an identification of $\mathbb{R}\operatorname{Hom}_{\mathcal{O}_{\mathcal{X}}}(E,M)$ with $\mathbb{R}\Gamma(\mathcal{X},E^\vee \otimes^{\mathbb{L}}_{\mathcal{O}_{\mathcal{X}}} M)$, and so, we have that $\operatorname{pd}(E) \leq \operatorname{pd}(\mathcal{O}_{\mathcal{X}}) = \operatorname{cd}(\mathcal{X})$. In particular, if $\mathcal{X}$ has finite cohomological dimension, then the result is trivial.
    
    Next, if $\mathcal{X}$ has no closed points of positive characteristic, then we claim that $\mathcal{X}$ is a $\mathbf{Q}$-stack, and so, $\mathcal{X}$ has finite cohomological dimension by \cite[Theorem C]{Hall/Rydh:2015}. For each integer $n$, let $\mathcal{X}_{n}$ denote algebraic stack associated to $\mathcal{X} \times_{\operatorname{Spec}( \mathbf{Z})} \operatorname{Spec} (\mathbf{Z}[n^{-1}] )\hookrightarrow \mathcal{X}$. Since every closed point of $\mathcal{X}$ has characteristic zero, it follows that every closed point of $\mathcal{X}$ factors through $\mathcal{X}_{n} \hookrightarrow \mathcal{X}$, and so, this map is consequently an isomorphism. However, $\mathcal{X}_{\mathbf{Q}} = \varprojlim_n \mathcal{X}_n $ is isomorphic to $\mathcal{X}$, ensuring that $\mathcal{X} \to \operatorname{Spec}(\mathbf{Z})$ factors through $\operatorname{Spec}(\mathbf{Q}) \to \operatorname{Spec} (\mathbf{Z})$; that is, $\mathcal{X}$ is a $\mathbf{Q}$-stack as desired.

    In general, as $\mathcal{X}$ has the compact resolution property, it follows that $D_{\operatorname{qc}}(\mathcal{X})$ is compactly generated \cite[Theorem 8.4]{Hall/Rydh:2017}. If $x \colon \operatorname{Spec} (k) \to \mathcal{X}$ is a positive characteristic geometric point with (affine) automorphism group $G_x$ over $k$, then $(G_x)^0_{\mathrm{red}}$ is a torus \cite[Theorem 1.1]{Hall/Neeman/Rydh:2019}. Arguing as in the proof of \cite[Theorem 1.12(5)]{Alper/Hall/Halpern-Leistner/Rydh:2024}, there is a quasi-finite flat, affine, syntomic morphism $f\colon \mathcal{W} \to \mathcal{X}$, where $\mathcal{W}$ is nicely fundamental and $j \colon \mathcal{U}=f(\mathcal{W}) \subseteq \mathcal{X}$ is a quasi-compact open that contains all of the points of positive characteristic. Let $i \colon \mathcal{Z} \hookrightarrow \mathcal{X}$ be a finitely presented closed immersion with $|\mathcal{X}|\setminus |i(\mathcal{Z})| = |\mathcal{U}|$. Note that every closed point of $\mathcal{Z}$ has characteristic $0$. Hence, there exists an $r_{\mathcal{Z}}$ such that if $F$ is a compact vector bundle on $\mathcal{Z}$, then $\operatorname{pd}(F) \leq r_{\mathcal{Z}}$. By \Cref{lem:bound-strata}, it suffices to treat the case where $\mathcal{U}=\mathcal{X}$; that is, there is a quasi-finite faithfully flat, affine, syntomic morphism $f\colon \mathcal{W} \to \mathcal{X}$, where $\mathcal{W}$ is nicely fundamental.

    Let $\mathbf{E}$ be the category of stacks $\mathcal{V}$ that are quasi-finite, flat, of finite presentation, representable, and separated over $\mathcal{X}$. Let $\mathbf{F}$ be the subcategory of $\mathbf{E}$ consisting of those objects $\mathcal{V}$ such for all \'{e}tale morphisms $\mathcal{V}^\prime \to \mathcal{V}$ in $\mathbf{E}$ there exists an integer $r^\prime$ such that $\operatorname{pd}(F^\prime) \leq r^\prime$ for all compact vector bundles $F'$ on $\mathcal{V}^\prime$. Since $\mathcal{W}$ has cohomological dimension $0$, for every morphism $\mathcal{W}^\prime \to \mathcal{W}$ in $\mathbf{E}$, $\mathcal{W}^\prime$ has finite cohomological dimension, and so, $\mathcal{W}$ is in $\mathbf{F}$. It remains to prove that $\mathcal{X}$ is in $\mathbf{F}$, which we will accomplish using \cite[Theorem D']{Hall/Rydh:2018}. Thus, we need to verify the following: 
    \begin{enumerate}[label=($\text{D}$\arabic*)]
        \item\label{D1a} if $(\mathcal{V}^\prime\to\mathcal{V})$ is an \'{e}tale morphism in $\mathbf{E}$ and $\mathcal{V}$ is in $\mathbf{F}$, then $\mathcal{V}^\prime$ is in $\mathbf{F}$.   
        \item\label{D2a} If $(\mathcal{Y}^\prime\xrightarrow{y} \mathcal{Y})$ is a morphism in $\mathbf{E}$ which is finite, flat and surjective, and $\mathcal{Y}^\prime $ is in $\mathbf{F}$, then $\mathcal{Y}$ belongs to $\mathbf{F}$.
        \item\label{D3a} If $(\mathcal{S}\xrightarrow{i} \mathcal{V})$, $( \mathcal{V}^\prime\xrightarrow{v}\mathcal{V})$ are morphisms in $\mathbf{E}$ where $i$ is an open immersion and $v$ is \'{e}tale and an isomorphism on $\mathcal{V}\setminus \mathcal{S}$, then $\mathcal{S}$ and $\mathcal{V}^\prime$ being objects of $\mathbf{F}$ ensures that $\mathcal{V}$ is as well.
    \end{enumerate}
    Clearly, \ref{D1a} is trivial and \ref{D3a} follows easily from the Mayer-Vietoris triangle (e.g., \cite[Lemma 5.9(1)]{Hall/Rydh:2017}). For \ref{D2a}: if $E$ is a compact vector bundle on $\mathcal{Y}$, set $E^{(0)} = E$ and inductively define $E^{(i)}$ via the short exact sequence of compact vector bundles $0 \to E^{(i)} \to y_\ast y^\times E^{(i-1)} \to E^{(i-1)} \to 0$ on $\mathcal{Y}$. By assumption, there is an $r^\prime$ such that $r^\prime \geq \operatorname{cd}(F^\prime)$ for all compact vector bundles on $\mathcal{Y}^\prime$. Then if $M$ is a quasi-coherent $\mathcal{O}_{\mathcal{Y}}$-module, it follows that for every $r^{\prime \prime} > r^\prime$: 
    \begin{displaymath}
        \operatorname{Ext}^{r^{\prime \prime}}_{\mathcal{O}_{\mathcal{Y}}}(E^{(i)},M) \cong \operatorname{Ext}^{r^{\prime \prime}+1}_{\mathcal{O}_{\mathcal{Y}}}(E^{(i-1)},M). 
    \end{displaymath}
    Working inductively, we see that if $r^{\prime \prime} > r^\prime$: 
    \begin{displaymath}
        \operatorname{Ext}^{r^{\prime \prime}}_{\mathcal{O}_{\mathcal{Y}}}(E^{(i)},M) \cong \operatorname{Ext}^{r^{\prime \prime}+i}_{\mathcal{O}_{\mathcal{Y}}}(E,M). 
    \end{displaymath}
    Since $E$ is compact, there exists $e$ such that $\operatorname{Ext}^{e^\prime}_{\mathcal{O}_{\mathcal{Y}}}(E,M) =0$, whenever $e^\prime >  e$. It follows that $\operatorname{Ext}^{r^{\prime \prime}}_{\mathcal{O}_{\mathcal{Y}}}(E^{(i)},M) =0$, whenever $i+r^{\prime \prime} > e$. In particular, $\operatorname{Ext}^{r^{\prime \prime}}_{\mathcal{O}_{\mathcal{Y}}}(E^{(e)},M) =0$ for all $r^{\prime \prime}> r^\prime$. That is, $\operatorname{pd}(E^{(e)}) \leq r^\prime$. Since $\operatorname{pd}(y_\ast y^\times E^{(i-1)}) \leq r^\prime$ it follows from \Cref{lem:bounded}\eqref{LI:bounded:2of3} that $\operatorname{pd}(E^{(e-1)}) \leq r^\prime$. Working inductively, we conclude that $\operatorname{pd}(E) \leq r^\prime$, as claimed.    

    For the latter claim, we note that compact generation of $D_{\operatorname{qc}}(\mathcal{X})$ implies that there is an equivalence of $D(\operatorname{Qcoh}(\mathcal{X}))$ with $D_{\operatorname{qc}}(\mathcal{X})$. Now apply \cite[\href{https://stacks.math.columbia.edu/tag/08HP}{Tag 07K7}]{StacksProject} to $\operatorname{Hom}_{\mathcal{O}_{\mathcal{X}}}(E,-) \colon \operatorname{Qcoh}(\mathcal{X}) \to \operatorname{Ab}$, which completes the proof.
\end{proof}

\begin{remark}
    The proof of \ref{D2} above shows that if $\mathcal{G}$ is a gerbe for a finite flat group scheme over an affine scheme, then every compact vector bundle on $\mathcal{G}$ is projective. A simple argument then shows that if $\mathcal{G} = BG$, where $G$ is a finite flat group scheme over a ring $A$, then every compact vector bundle is a direct summand of a direct sum of copies of the regular representation.
\end{remark}
We now come to the main result of this section.
\begin{proposition}\label{prop:approximation_resolution_property}
    If $\mathcal{X}$ is a quasi-compact and quasi-separated algebraic stack with the compact resolution property, then approximation by compact complexes holds for $\mathcal{X}$.
\end{proposition}

\begin{proof}
    Choose a $d$ as in \Cref{thm:uniform-bound-ext-dim}. Let $T$ be a closed subset of $|\mathcal{X}|$ with quasi-compact open complement $\mathcal{U}$. Note that the immersion $j\colon\mathcal{U}\to \mathcal{X}$ has cohomological dimension $\leq k$ for some $k$. Set $r=k+d$. We claim that approximation holds for all $(T, E, m)$ where $E$ is $(m-r)$-pseudocoherent and $H^{i}(E)$ is supported on $T$ for all $i\geq m-r$. 

    Let $t$ be the largest integer such that $H^{t}(E)$ is nonzero. We prove the claim by induction on $t$. If $t< m$, then $0 \to E$ is an approximation and we are done. If $t\geq m$, then $H^{t}(E)$ is a quasi-coherent sheaf on $\mathcal{X}$ of finite type. Since $\mathcal{X}$ has the compact resolution property, $D_{\operatorname{qc}}(\mathcal{X})$ is compactly generated \cite[Proposition 8.4]{Hall/Rydh:2017}, and so, $D(\operatorname{Qcoh}(\mathcal{X}))$ is triangle equivalent to $D_{\operatorname{qc}}(\mathcal{X})$ \cite[Theorem 1.2]{Hall/Neeman/Rydh:2019}. Hence, $E$ is quasi-isomorphic to a complex $(\cdots \xrightarrow{\partial^{s-1}} E^s \xrightarrow{\partial^s} E^{s+1} \xrightarrow{\partial^{s+1}}\cdots)$ of quasi-coherent $\mathcal{O}_{\mathcal{X}}$-modules, where $E^s = 0$ for $s>t$. Consider now the surjection $\ker(\partial^t) \twoheadrightarrow H^t(E)$; then the compact resolution property implies that there is a vector bundle $F$ and a morphism $F \to \ker(\partial^t)$ such that the induced morphism $F \to H^t(E)$ is surjective. It follows that there is an induced map $\psi\colon F[-t] \to E$ in $D_{\operatorname{qc}}(\mathcal{X})$ such that $H^t(\psi)$ is surjective. By \Cref{thm:uniform-bound-ext-dim}, we have since $t\geq m$ that
    \begin{displaymath}
        \begin{aligned}
            \operatorname{Hom}_{\mathcal{O}_{\mathcal{X}}}(F[-t],\mathbb{R} j_\ast \mathbb{L} j^\ast E) &\cong H^t(\mathbb{R}\operatorname{Hom}_{\mathcal{O}_{\mathcal{X}}}(F,\mathbb{R} j_\ast \mathbb{L} j^\ast E)) \\
            &\cong H^t(\tau^{\geq m}\mathbb{R}\operatorname{Hom}_{\mathcal{O}_{\mathcal{X}}}(F,\mathbb{R} j_\ast \mathbb{L} j^\ast E)) \\
            &\cong H^t(\tau^{\geq m}\mathbb{R}\operatorname{Hom}_{\mathcal{O}_{\mathcal{X}}}(F,\tau^{\geq m-d}\mathbb{R} j_\ast \mathbb{L} j^\ast E)) \\
            &\cong H^t(\tau^{\geq m}\mathbb{R}\operatorname{Hom}_{\mathcal{O}_{\mathcal{X}}}(F,\tau^{\geq m-d}\mathbb{R} j_\ast \tau^{\geq m-r}\mathbb{L} j^\ast E)) \\
            &= 0.
            \end{aligned}
    \end{displaymath}
    Hence, $\mathbb{L} j^\ast (\psi)$ vanishes, and so, \Cref{prop:nonnoetherian-deligne-formula} implies that there is a factorization of $\psi$ through a morphism $\phi \colon Q=F[-t] \otimes^{\mathbb{L}}_{\mathcal{O}_{\mathcal{X}}}\mathcal{Q}_n \to E$ and $H^t(\phi)$ is surjective. Since $F$ is compact and $\mathbb{L}j^\ast \mathcal{Q}_n$ is the zero object, it follows that $Q$ belongs to $D_{\operatorname{qc}}(\mathcal{X})^c \cap D_{\operatorname{qc},T}(\mathcal{X})$. Now form the distinguished triangle
    \begin{displaymath}
        Q \xrightarrow{\phi} E \to E^\prime \to Q[1].
    \end{displaymath}
    Then $H^i(E^\prime) = 0$ for all $i\geq t$, $E^\prime$ is $(m-r)$-pseudocoherent and $H^i(E^\prime)$ is supported on $T$ for all $i\geq m-r$. By induction, there is an object $P^\prime$ of $D_{\operatorname{qc}}(\mathcal{X})^c \cap D_{\operatorname{qc},T}(\mathcal{X}) $ and an $m$-isomorphism $\pi^\prime \colon P^\prime \to E^\prime$. Putting this all together, we obtain a diagram of distinguished triangles:
    \begin{displaymath}
        \xymatrix{Q \ar@{=}[d] \ar[r] & P \ar[d]_{\pi} \ar[r] & P^\prime \ar[d]^{\pi^\prime} \ar[r] & Q[1] \ar@{=}[d] \\ Q \ar[r]_{\psi} & E \ar[r] & E^\prime \ar[r] & Q[1].}
    \end{displaymath}
    Then $\pi \colon P\to E$ is the desired approximation for $(T,E,m)$.
\end{proof}

\section{Descent of approximation by compact complexes}
\label{sec:approximation_by_perfect}

In this section, we show that approximation of pseudocoherent complexes by compact complexes descends along quasi-finite flat covers. The key ingredient is quasi-finite flat d\'{e}vissage for algebraic stacks \cite[Theorem\ $\text{D}^\prime$]{Hall/Rydh:2018}. We first show that approximation descends along finite flat coverings. 

\begin{proposition}\label{prop:finite_flat_decent_approximation}
    Let $\mathcal{X}$ be a quasi-compact and quasi-separated algebraic stack. Let $f\colon\mathcal{Y}\to \mathcal{X}$ be a morphism that is finite, faithfully flat, and of finite presentation. If approximation by compact complexes holds for $\mathcal{Y}$, then it holds for $\mathcal{X}$. 
\end{proposition}

\begin{proof}
    Let $T$ be a closed subset of $|\mathcal{X}|$ with quasi-compact complement. Then $f^{-1}(T)$ is also closed in $|\mathcal{Y}|$ with quasi-compact complement. By assumption, there exists an integer $r$ such that approximation holds for all triples $(f^{-1}(T), G, m)$ satisfying the conditions in \Cref{def:approximation}. We claim the same integer $r$ works for $(T, E, m)$ on $\mathcal{Y}$. Let $t$ be the largest integer such that $H^{t}(E)\neq 0$. We prove the claim by induction on $t$. If $t< m-r$, then $\tau^{\geq m-r}E$ is the zero object, and so, $0 \to E$ is an $m$-isomorphism, which completes the base case of the induction. Now suppose $t\geq m-r$. We first note that $\mathbb{R}f_{\ast}$ admits a conservative t-exact right adjoint $f^{\times}(-) \cong f^\times(\mathcal{O}_{\mathcal{X}}) \otimes^{\mathbb{L}}_{\mathcal{O}_{\mathcal{Y}}} \mathbb{L}f^\ast(-)$, where $f^\times(\mathcal{O}_{\mathcal{X}})$ is a vector bundle of finite rank \cite[Theorem 4.14 \& Corollary 4.15]{Hall/Rydh:2017}. It follows that $f^{\times}E$ is a $(m-r)$-pseudocoherent complex and $H^{i}(f^{\times}E)$ is supported on $f^{-1}(T)$ for $i\geq m-r$. By assumption, there exists a compact complex $P$ supported on $f^{-1}(T)$ and an $m$-isomorphism map $P\to f^{\times}E$. By adjunction, we obtain a distinguished triangle: 
    \begin{displaymath}
        \mathbb{R}f_{\ast}P\to E\to E^\prime\to \mathbb{R}f_{\ast}P[1].
    \end{displaymath}
    Then we claim that $E^\prime$ is $(m-r)$-pseudocoherent, $H^{i}(E^\prime)$ is supported on $T$ if $i\geq m-r$, and $H^{t}(E^\prime)=0$ if $i\geq t$. Indeed, $\mathbb{R}f_\ast P$ is compact and so pseudocoherent and $E$ is $(m-r)$-pseudocoherent, so $E^\prime$ is $(m-r)$-pseudocoherent. The statement on support of the cohomology sheaves is similar. Finally, to see that $H^t(E^\prime) = 0$, we note that there is an exact sequence of cohomology:
    \begin{displaymath}
        H^t(\mathbb{R}f_\ast P) \to H^t(E) \to H^t(E^\prime) \to H^{t+1}(\mathbb{R}f_\ast P).
    \end{displaymath}
    Now $f$ is finite and affine, so $H^i(\mathbb{R}f_\ast P) $ is isomorphic to $f_\ast H^i(P)$ for all integers $i$. Since $f$ is flat, $H^{t+1}(\mathbb{L} f^\ast E) = f^\ast H^{t+1}(E) = 0$ and $P \to \mathbb{L} f^\ast E$ is an $m$-isomorphism, it follows that $H^{t+1}(P) =0$. Hence, it remains to prove that $H^t(\mathbb{R}f_\ast P) \to H^t(E)$ is surjective, which is immediate from the above and the surjectivity of $H^i(\mathbb{R} f_\ast f^\times E) \to H^i(E)$. 
    It now follows that by induction, there exists an approximation $Q\to E^\prime$ for $(E^\prime, T,m)$. Consider the following distinguished triangle$$\mathbb{R}f_{\ast}P\to K\to Q\to \mathbb{R}f_{\ast}P[1],$$where the last map is the composition $Q\to E^\prime\to \mathbb{R}f_{\ast}P[1]$. Since $\mathbb{R}f_{\ast}P$ and $Q$ are compact and supported on $T$, so is $K$. By TR3, we obtain the following diagram of distinguished triangles
    \begin{equation}
        \begin{tikzcd}
            \mathbb{R}f_{\ast}P\arrow[r]\arrow[d,equal] & K\arrow[r]\arrow[d] & Q\arrow[r]\arrow[d] & \mathbb{R}f_{\ast}P[1]\arrow[d,equal]\\
        \mathbb{R}f_{\ast}P\arrow[r] & E\arrow[r] & E^\prime\arrow[r] & \mathbb{R}f_{\ast}P[1].
        \end{tikzcd}
    \end{equation}
    The map $K\to E$ is the desired approximation for $(E,T,m)$.
\end{proof}

We recall some notions that we will use next. Consider a cartesian diagram of quasi-compact and quasi-separated algebraic stacks:
\begin{equation}
    \xymatrix{\mathcal{U}^\prime \ar[r]^{j^\prime}\ar[d]_{f_{\mathcal{U}}} & \mathcal{X^\prime} \ar[d]^{f} \\ \mathcal{U} \ar[r]_j & \mathcal{X},}\label{eq:MV}    
\end{equation}
where $j$ is an open immersion. Now choose a finitely presented closed immersion  $i\colon \mathcal{Z} \hookrightarrow \mathcal{X}$ with $|\mathcal{X}|\setminus |i(\mathcal{Z})| = |\mathcal{U}|$. If $f$ is flat, concentrated and $f_{\mathcal{Z}} \colon \mathcal{Z}^\prime =\mathcal{X}^\prime \times_{\mathcal{X}} \mathcal{Z} \to \mathcal{Z}$ is an isomorphism, then we say that the square above is a \textbf{flat Mayer-Vietoris  square} \cite{Hall/Rydh:2023}. The key property here is that $f$ induces an adjoint t-exact equivalence \cite[Theorem 4.2]{Hall/Rydh:2023}:
\begin{equation}
    D_{\operatorname{qc},\mathcal{Z}}(\mathcal{X}) \cong D_{\operatorname{qc},\mathcal{Z}^\prime}(\mathcal{X}^\prime). \label{eq:fmv-equiv}
\end{equation}
Moreover, if $P$ is in $D_{\operatorname{qc},\mathcal{Z}^\prime}(\mathcal{X}^\prime) \cap D_{\operatorname{qc}}(\mathcal{X}^\prime)^c$, then $\mathbb{R}f_\ast P$ is in $D_{\operatorname{qc},\mathcal{Z}}(\mathcal{X}) \cap D_{\operatorname{qc}}(\mathcal{X})^c$ \cite[Lemma 5.9]{Hall/Rydh:2017}. If $f$ is \'{e}tale, then the square \eqref{eq:MV} is an \textbf{\'{e}tale neighbourhood} (cf.~\cite{Rydh:2011}).

\begin{proposition}\label{prop:descending_approximation_along_etale_nbhd}
    Consider a flat Mayer--Vietoris square as in \eqref{eq:MV}. Assume that $\mathcal{X}^\prime$ and $\mathcal{U}$ satisfy the Thomason condition. If approximation by compact complexes holds for $\mathcal{U}$ and $\mathcal{X}^\prime$, then it holds for $\mathcal{X}$. 
\end{proposition}

\begin{proof}
    By replacing $\mathcal{X}^\prime$ with $\mathcal{X}^\prime \amalg \mathcal{U}$, we may assume that $f$ is faithfully flat. Let $T$ be a closed subset of $|\mathcal{X}|$ with quasi-compact complement $V$. By assumption, there exists an integer $r_{\mathcal{U}}$ such that approximation holds for any triple $(T\cap|\mathcal{U}|, G, m)$ on $\mathcal{U}$, where $G$ is $(m-r_{\mathcal{U}})$-pseudocoherent on $\mathcal{U}$ and $H^i(E)$ is supported on $T$ for $i\geq m-r_{\mathcal{U}}$. Let $T^\prime=T\setminus |\mathcal{U}|$. By assumption, $f$ restricts to an isomorphism $f^{-1}(T^\prime)\to T^\prime$ and its complement in $\mathcal{X}^\prime$ is also quasi-compact. Then there exists an integer $r^\prime$ such that approximation holds for all $(f^{-1}(T^\prime), G^\prime, m)$ on $\mathcal{X}^\prime$, where $G^\prime$ is $(m-r^\prime)$-pseudocoherent on $\mathcal{X}^\prime$ and $H^i(G^\prime)$ is supported on $T^\prime$ for $i\geq m-r^\prime$. Set $r=\max(r_{\mathcal{U}},r^\prime)$. We will show that approximation holds for all triples $(T,E,m)$ on $\mathcal{X}$, where $E$ is $(m-r)$-pseudocoherent on $\mathcal{X}$ and $H^i(E)$ is supported on $T$ for $i\geq m-r$.

    Consider the triple $(T\cap|\mathcal{U}|, \mathbb{L}j^{\ast}E,m)$. Since $r_{\mathcal{U}} \leq r$, $\mathbb{L}j^\ast E$ is $(m-r_{\mathcal{U}})$-pseudocoherent and $H^i(\mathbb{L}j^\ast E)$ is supported on $T \cap \mathcal{U}$ for $i\geq m-r_{\mathcal{U}}$. By assumption there is a compact complex $P_{\mathcal{U}}$ on $\mathcal{U}$ supported on $T \cap |\mathcal{U}|$ and an $m$-isomorphism  $P_{\mathcal{U}}\to \mathbb{L}j^{\ast}E$. Note that we have the following localization sequence \cite[Lemma 5.12]{Hall/Rydh:2017}:
    \begin{displaymath}
        D_{\operatorname{qc},f^{-1}(T^\prime)}(\mathcal{X}^\prime)\to D_{\operatorname{qc},f^{-1}(T)}(\mathcal{X}^\prime)\to D_{\operatorname{qc},f^{-1}(T) \cap \mathcal{U}^\prime}(\mathcal{U}^\prime).
    \end{displaymath}
    By Thomason's localization theorem \cite[Lemma 3.3, Theorem 3.12, Lemma 6.7]{Hall/Rydh:2017}, there exists a compact complex $P^\prime$ on $\mathcal{X}^\prime$ supported on $f^{-1}(T)$ such that $\mathbb{L}j^{\prime, \ast }P^\prime$ is isomorphic to $\mathbb{L}f_{\mathcal{U}}^{\ast}(P_{\mathcal{U}}\oplus P_{\mathcal{U}}[d])$, where $d\geq 0$ is odd and $\tau^{\geq m+d} P^\prime \cong 0$. This gives us the following composition
    \begin{displaymath}
        s_{\mathcal{U}^\prime}\colon \mathbb{L}j^{\prime, *}P^\prime\cong \mathbb{L}f_{\mathcal{U}}^{\ast}(P^\prime\oplus P^\prime[d])\to \mathbb{L}f_{\mathcal{U}}^{\ast}P^\prime\to \mathbb{L}f_{\mathcal{U}}^{\ast}\mathbb{L}j^{\ast}E\cong\mathbb{L}j^{\prime,\ast}\mathbb{L}f^{\ast}E.
    \end{displaymath}
    By Thomason's localization theorem \cite[Theorem 2.1.5]{Neeman:1996} applied to the localization sequence above, there exist morphisms $P^\prime \xleftarrow{\lambda} P^{\prime \prime} \xrightarrow{s^{\prime \prime}} \mathbb{L}f^\ast E$, where $P^{\prime \prime}$ is compact on $\mathcal{X}^\prime$ and supported on $f^{-1}(T)$ such that the cone $L$ of $\lambda$ satisfies $\mathbb{L}j^{\prime,\ast}L \cong 0$ and $\mathbb{L}j^{\prime,\ast} s^{\prime \prime} = s_{\mathcal{U}^\prime} \circ \mathbb{L}j^{\prime,\ast}\lambda$.

    By \cite[Lemma 5.9(3),(4)]{Hall/Rydh:2017}, there is a compact complex $P$ on $\mathcal{X}$ with support in $T$ such that $\mathbb{L}f^\ast P \cong P^{\prime \prime}$ and $\mathbb{L}j^\ast P \cong P_{\mathcal{U}}\oplus P_{\mathcal{U}}[d]$. The induced Mayer-Vietoris  exact sequence from \cite[Lemma 5.9(1)]{Hall/Rydh:2017} gives an exact sequence of abelian groups:
    \begin{displaymath}
        \begin{aligned}
            \operatorname{Hom}_{\mathcal{O}_{\mathcal{X}}}(P,E)& \to\operatorname{Hom}_{\mathcal{O}_{\mathcal{U}}}(\mathbb{L}j^{\ast}P, \mathbb{L}j^{\ast}E)\oplus\operatorname{Hom}_{\mathcal{O}_{\mathcal{X}^\prime}}(\mathbb{L}f^{\ast}P,\mathbb{L}f^{\ast}E)\\&\to\operatorname{Hom}_{\mathcal{O}_{\mathcal{U}^\prime}}(\mathbb{L}f_{\mathcal{U}}^{\ast}\mathbb{L}j^{\ast}P,\mathbb{L}f_{\mathcal{U}}^{\ast}\mathbb{L}j^{\ast}E),
        \end{aligned}
    \end{displaymath}
    and so, our maps $P_{\mathcal{U}} \oplus P_{\mathcal{U}}[d] \to \mathbb{L} j^{\ast}E$ and $P^{\prime \prime} \to \mathbb{L} f^\ast E$ give rise to a map $P \to E$. Now form the distinguished triangle:
    \begin{displaymath}
        P \to E \to E^\prime \to P[1].
    \end{displaymath}
    Then after restricting to $\mathcal{U}$ we obtain a diagram with distinguished rows and columns:
    \begin{displaymath}
        \xymatrix{P_{\mathcal{U}}[d]\ar[r] \ar[d] & 0 \ar[d]\ar[r] & \ar[d] P_{\mathcal{U}}[d+1] \ar@{=}[r] & P_{\mathcal{U}}[d+1] \ar[d] \\
        P_{\mathcal{U}} \oplus P_{\mathcal{U}}[d] \ar[d] \ar[r] & \mathbb{L}j^{\ast}E \ar@{=}[d]\ar[r] & \ar[d] \mathbb{L}j^{\ast}E^\prime \ar[r] & (P_{\mathcal{U}} \oplus P_{\mathcal{U}}[d])[1] \ar[d] \\
        P_{\mathcal{U}}\ar[r] & \mathbb{L}j^{\ast}E \ar[r] & E^{\prime \prime} \ar[r] & P_{\mathcal{U}}[1].}
    \end{displaymath}
    Now $\tau^{\geq m}E^{\prime \prime}$ is the zero object by construction, and $\tau^{\geq m}(P_{\mathcal{U}}[d+1])$ is also the zero object because $d\geq 0$ is chosen so that $\tau^{\geq m+d}P_{\mathcal{U}} \cong 0$. It follows that $\tau^{\geq m}\mathbb{L}j^\ast E^\prime$ is the zero object. Then $H^i(E^\prime)$ is supported on $T^\prime$ if $i\geq m-r^\prime \geq m-r$
    and so there exists a distinguished triangle
    \begin{displaymath}
        \xymatrix{F^\prime \ar[r] & \mathbb{L}f^{\ast}E^\prime \ar[r] & H^\prime \ar[r] & F[1]}
    \end{displaymath}
    on $\mathcal{X}^\prime$, where $F^\prime$ is a compact complex supported on $T^\prime$ and $\tau^{\geq m}H^\prime $ is the zero object. Since the square is a flat Mayer-Vietoris  square, there is an induced morphism $\mathbb{R}f_\ast F^\prime \to E^\prime$ such that the pullback is the composition $\mathbb{L}f^\ast \mathbb{R}f_\ast F^\prime \cong F^\prime \to \mathbb{L}f^\ast E^\prime$ \cite[Lemma 5.9(2)]{Hall/Rydh:2017}. Moreover, $F=\mathbb{R}f_{\ast}F^\prime$ is a compact complex on $\mathcal{X}$ supported on $T^\prime$ \cite[Lemma 5.9(4)]{Hall/Rydh:2017}. Now form the distinguished triangle:
    \begin{displaymath}
        \xymatrix{F \ar[r] & E^\prime \ar[r] & H \ar[r] & F[1].}
    \end{displaymath}
    Then $\mathbb{L}f^\ast H \cong H^\prime$ and so the flatness of $f$ implies that $0 \cong \tau^{\geq m}\mathbb{L}f^\ast H \cong \mathbb{L}f^\ast (\tau^{\geq m}H)$. Since $f$ is faithfully flat, $\tau^{\geq m} H \cong 0$. We now obtain a diagram with distinguished rows and columns:
    \begin{displaymath}
        \xymatrix{P \ar@{=}[d] \ar[r] & \tilde{P} \ar[r] \ar[d] & \ar[d] F \ar[r] & P[1] \ar@{=}[d]\\
    P \ar[r] \ar[d] & \ar[d] E \ar[r] & \ar[d] E^\prime \ar[r] & P[1] \ar[d] \\
    0 \ar[r] & H \ar@{=}[r] & H \ar[r] & 0.}
    \end{displaymath}
    Then $\tilde{P}$ is compact on $\mathcal{X}$, supported on $T$ and $\tau^{\geq m}H$ is the zero object, so $\tilde{P} \to E$ gives the desired approximation.
\end{proof}

We now establish the promised quasi-finite flat descent of approximation by compact complexes.

\begin{theorem}\label{thm:etale_descent_approximation}
    Let $f \colon \mathcal{X}^\prime \to \mathcal{X}$ be a morphism between quasi-compact and quasi-separated algebraic stacks that is representable, separated, quasi-finite, faithfully flat, and of finite presentation. Assume that for every quasi-compact, representable, separated, \'{e}tale morphism $\mathcal{W} \to \mathcal{X}^\prime$, $\mathcal{W}$ satisfies the Thomason condition and approximation by compacts holds. Then the same is true of $\mathcal{X}$.
\end{theorem}

\begin{proof}
    Let $\mathbf{E} = \mathbf{Stack}_{\text{rep,sep,qff,fp}/\mathcal{X}}$, whose objects are those $1$-morphisms $\mathcal{V} \to \mathcal{X}$ that are representable, separated, quasi-finite, flat, and of finite presentation.
    Let $\mathbf{App}$ be the full subcategory of $\mathbf{E}$ consisting of those $\mathcal{V}$, where for every \'{e}tale morphism $\mathcal{V}^\prime \to \mathcal{V}$ in $\mathbf{E}$, $\mathcal{V}^\prime$ satisfies the Thomason condition and approximation by compacts hold. We will apply \cite[Theorem\ $\text{D}^\prime$]{Hall/Rydh:2018} to $\mathbf{App}\subseteq\mathbf{E}$. To this end, we need to verify the following:
    \begin{enumerate}[label=($\text{D}$\arabic*)]
        \item\label{D1} if $(\mathcal{W}\to\mathcal{V})$ is an \'{e}tale morphism in $\mathbf{E}$ and $\mathcal{V}$ is in $\mathbf{App}$, then $\mathcal{W}$ is in $\mathbf{App}$.   
        \item\label{D2} If $(\mathcal{W}^\prime\to\mathcal{W})$ is a finite, flat and surjective morphism in $\mathbf{E}$, and $\mathcal{W}^\prime$ is in $\mathbf{App}$, then $\mathcal{W}$ is in $\mathbf{App}$.
        \item\label{D3} If $(\mathcal{W}\xrightarrow{i} \mathcal{V})$, $( \mathcal{V}^\prime\xrightarrow{v}\mathcal{V})$ are morphisms in $\mathbf{E}$, where $i$ is an open immersion and $v$ is \'{e}tale and an isomorphism on $\mathcal{V}\setminus \mathcal{W}$; then $\mathcal{W}$, $\mathcal{V}^\prime$ is in $\mathbf{App}$ implies that $\mathcal{V}$ is in $\mathbf{App}$.
    \end{enumerate}
    Now \ref{D1} is a tautology. Also, \ref{D2} and \ref{D3} follow from \Cref{prop:finite_flat_decent_approximation} and \Cref{prop:descending_approximation_along_etale_nbhd}, respectively (also see \cite[Theorem 6.9]{Hall/Rydh:2017} for the descent of the Thomason condition). This finishes the proof.
\end{proof}

It follows immediately from \Cref{thm:etale_descent_approximation} that we can prove \Cref{introthm:approximation_for_quasifinite} (on compact approximation for stacks with quasi-finite diagonal).
\begin{proof}[Proof of \Cref{introthm:approximation_for_quasifinite}]
      For (1), by \cite[Theorem\ 7.2]{Rydh:2011}, there is a quasi-finite, separated, and flat presentation $p\colon V\to \mathcal{X}$ from an affine scheme $V$. In particular, approximation by compact complexes holds for $V$ (e.g., \Cref{prop:approximation_resolution_property} or \cite[\href{https://stacks.math.columbia.edu/tag/08EQ}{Tag 08EQ}]{StacksProject}. Applying \Cref{thm:etale_descent_approximation} to $p$ finishes the proof. For (2), argue similarly to \cite[Theorem 7.4]{Hall/Rydh:2019}.
\end{proof}

Recall that an algebraic stack is of \textbf{s-global (compact) type} if it admits a separated \'{e}tale cover by a stack with the (compact) resolution property. See \cite{Rydh:2015} for the precise definition and examples for algebraic stacks of s-global type. An immediate consequence of \Cref{prop:approximation_resolution_property} and \Cref{thm:etale_descent_approximation} is the following.
\begin{corollary}\label{cor:approximation_for_good_moduli_spaces}
    If $\mathcal{X}$ is an algebraic stack of s-global compact type, then approximation by compact complexes holds for $\mathcal{X}$. 
\end{corollary}
A simple consequence is \Cref{introthm:approximation_for_good_moduli_spaces}.
\begin{proof}[Proof of \Cref{introthm:approximation_for_good_moduli_spaces}]
  By \cite[Theorem 5.7]{Rydh:2023}, $\mathcal{X}$ is of
  s-global compact type. Now apply \Cref{cor:approximation_for_good_moduli_spaces}.
\end{proof}

We also have the following refinement of \cite[Lemma A.4]{Hall:2022} and \cite[Lemma 3.18]{Lank/Olander:2024} (cf.\ \cite[\href{https://stacks.math.columbia.edu/tag/0CRW}{Tag 0CRW}]{StacksProject}).
\begin{corollary}
  Let $\pi \colon \mathcal{U} \to \mathcal{X}$ be a concentrated, flat,
  and finitely presented morphism of algebraic stacks. Form the
  distinguished triangle:
  \[
    \xymatrix{K \ar[r] & \mathcal{O}_{\mathcal{X}} \ar[r] & \mathbb{R}
      \pi_\ast \mathcal{O}_{\mathcal{U}} \ar[r] & K[1].}
  \]
  Let $M = K^{\otimes i}$ or $(\mathbb{R}\pi_\ast \mathcal{O}_{\mathcal{U}})^{\otimes i}$.
  Then
  $M \simeq \operatorname{hocolim}_{r} P_r$, where
  $P_r \in D^{\leq i\operatorname{cd}(\pi)}_{\operatorname{qc}}(X) \cap
  D_{\operatorname{qc}}(\mathcal{X})^c$, in the following
  situations:
  \begin{enumerate}
  \item \label{CI:hocolim_pres:countable} $\mathcal{X}$ is Noetherian
    and has approximation by compact complexes; or
  \item $\mathcal{X}$ is of s-global compact type and either $\pi$ is tame
    (e.g., representable) or $\pi$ has affine stabilizers and
    $\mathcal{X}$ is equicharacteristic.
  \end{enumerate}
\end{corollary}
\begin{proof}
  By absolute Noetherian approximation
  \cite{Rydh:2023}, tor-independent base change
  \cite[Cor.~4.13]{Hall/Rydh:2017} and 
  \Cref{cor:approximation_for_good_moduli_spaces}, it suffices to
  establish the result in case \eqref{CI:hocolim_pres:countable}. By
  \Cref{prop:approximation-countable}, it remains to prove
  that $\Gamma(V,\mathbb{R}^p \pi_\ast \mathcal{O}_{\mathcal{U}})$ is
  countably generated for every affine object $V$ of the lisse-\'etale
  topos of $\mathcal{X}$; now argue as in  \cite[Lemma A.4]{Hall:2022}.
\end{proof}

\section{D\'{e}vissage}
\label{sec:devissage}

\subsection{Descent}
\label{sec:devissage_descent}

This section explores the behavior of generation along the derived pushforward of a morphism between algebraic stacks. We will use the notion of descendability, which we will briefly recall. Our setting is for tensor triangulated categories, so we phrase things in that language. See \cite[Definition 3.18]{Mathew:2016} for the $\infty$-categorical analog. 
\begin{definition}[{\cite[Proposition 3.15 \& Definition 3.16]{Balmer:2016}}]\label{def:descendable_object}
    Let $(\mathcal{T},\otimes,\mathbf{1})$ be a tensor triangulated category. Suppose $A$ is a commutative monoid object $\mathcal{T}$. We say $A$ is \textbf{descendable}\footnote{\cite{Balmer:2016} imposes that $\mathcal{T}$ is essentially small, but this is not needed for our setting, see \cite[\S 6]{DeDeyn/Lank/KabeerManaliRahul:2025} for details.} if the smallest thick tensor-ideal containing $A$ coincides with $\mathcal{T}$. Equivalently, $A$ is descendable if and only if in a distinguished triangle $K \xrightarrow{\delta} \mathbf{1} \to A \to K[1]$ there exists $n$ such that $\delta^{\otimes n} = 0$.
\end{definition}
We say that a quasi-compact and quasi-separated morphism of algebraic stacks $f\colon \mathcal{Y} \to \mathcal{X}$ is \textbf{descendable} if $\mathbb{R}f_\ast \mathcal{O}_{\mathcal{Y}}$ is a descendable commutative monoid object in $D_{\operatorname{qc}}(\mathcal{X})$. 
\begin{remark}\label{rem:descendable-concentrated}
    If $f \colon \mathcal{Y} \to \mathcal{X}$ is a quasi-compact, quasi-separated and concentrated morphism of algebraic stacks, then $f$ is descendable if and only if there exists a $n\geq 0$ such that $D_{\operatorname{qc}}(\mathcal{X}) = \langle \mathbb{R}f_\ast D_{\operatorname{qc}}(\mathcal{Y})\rangle_n$. Indeed, choose $n\geq 0$ such that in the defining triangle $K \xrightarrow{\delta} \mathcal{O}_{\mathcal{X}} \to \mathbb{R}f_{\ast} \mathcal{O}_{\mathcal{Y}} \to K[1]$ we have $\delta^{\otimes n} = 0$. By \cite[Equation (6.5)]{Hall:2022} and the projection formula \cite[Corollary 4.12]{Hall/Rydh:2017}, it follows that:
    \[
        D_{\operatorname{qc}}(\mathcal{X})=\langle \mathbb{R} f_{\ast} \mathcal{O}_{\mathcal{Y}} \otimes D_{\operatorname{qc}}(\mathcal{X}) \rangle_n = \langle \mathbb{R} f_{\ast} D_{\operatorname{qc}}(\mathcal{Y}) \rangle_n.
    \]  
\end{remark}

We now prove \Cref{introthm:descendable_submersive}, which is a simple generalization of \cite[Theorem
B.2]{Hall/Priver:2024} (finite case with tame target), \cite[Theorem 7.1]{Hall:2022} (faithfully flat case) and \cite[Proposition 11.25]{Bhatt/Scholze:2017} (scheme case).

\begin{proof}[Proof of \Cref{introthm:descendable_submersive}]
    We first address the descendability. Let $p \colon \operatorname{Spec} (A) \to \mathcal{S}$ be a smooth
    covering and form the $2$-cartesian square:
    \[
    \xymatrix{\mathcal{Y}_A \ar[r]^{p_{\mathcal{Y}}} \ar[d]_{f_A} & \mathcal{Y} \ar[d]^{f} \\ \operatorname{Spec}(A)  \ar[r]_{p} & \mathcal{S}}
    \]
    By \cite[Theorem 7.1]{Hall:2022},
    $p$ is descendable. 
    Hence, it suffices to prove that
    $f_A$ is descendable \cite[Proposition 3.24(2)]{Mathew:2016} and so we may assume that $\mathcal{S}$ is an
    affine Noetherian scheme. Let
    $\operatorname{Spec} (B) \to \mathcal{Y}$ be a smooth covering. Then
    the induced composition
    $\beta \colon \operatorname{Spec} (B) \to \operatorname{Spec} (A)$ is a
    universally submersive morphism of Noetherian affine schemes. We have, from \cite[Proposition 11.25]{Bhatt/Scholze:2017}, that $\beta$ is descendable \cite[Proposition 11.25]{Bhatt/Scholze:2017}. Once more, by
    \cite[Proposition 3.24(2)]{Mathew:2016}, $f$ is
    descendable. 
\end{proof}

\begin{corollary}\label{cor:descendable_proper}
    If $f \colon \mathcal{Y} \to \mathcal{S}$ is a universally cohomologically proper and surjective morphism of Noetherian and concentrated algebraic stacks, then there exists an $n\geq 0$ such that $D^?_{\operatorname{coh}}(\mathcal{S}) = \langle \mathbb{R} f_\ast D^?_{\operatorname{coh}}(\mathcal{Y}) \rangle_n$, where $? \in \{b,-\}$.
\end{corollary}
\begin{proof}
    By \Cref{lem:concentrated_composition_inner_most}, $f$ is concentrated. It follows from \Cref{introthm:descendable_submersive}, \Cref{rem:descendable-concentrated} and \cite[Equation (6.5)]{Hall:2022} that $D^-_{\operatorname{coh}}(\mathcal{S}) = \langle \mathbb{R} f_\ast D^-_{\operatorname{coh}}(\mathcal{Y}) \rangle_n$ for some $n\geq 0$. It remains to prove that if $M \in \langle \mathbb{R} f_\ast D^-_{\operatorname{coh}}(\mathcal{Y}) \rangle_n \cap D^b_{\operatorname{coh}}(\mathcal{S})$, then $M \in \langle \mathbb{R} f_\ast D^b_{\operatorname{coh}}(\mathcal{Y}) \rangle_n$, which is just \cite[Lemma 3.10]{Dey/Lank:2024} (note that this is for schemes, but the argument is identical).
\end{proof}

\begin{corollary}\label{cor:bound_rouquier_dimension}
    Let $f \colon \mathcal{Y} \to \mathcal{S}$ be a universally cohomologically proper and surjective morphism of Noetherian and concentrated algebraic stacks. If $\mathcal{S}$ satisfies approximation by compacts, then 
    \begin{displaymath}
        \dim D^b_{\operatorname{coh}}(\mathcal{S}) \leq (\dim D^b_{\operatorname{coh}}(\mathcal{Y})+ 1) \cdot \min\{ \operatorname{level}^{\mathbb{R}f_\ast G} (\mathcal{O}_\mathcal{S}) : G \in D^b_{\operatorname{coh}}(\mathcal{Y}) \} -1.
    \end{displaymath}
    In particular, if $G$ is an object of $D^b_{\operatorname{coh}}(\mathcal{X})$ and $n\geq 0$ such that $\langle G \rangle_n = D^b_{\operatorname{coh}}(\mathcal{X})$, then $\langle \mathbb{R}f_\ast G \rangle_{nL}=D^b_{\operatorname{coh}}(\mathcal{S})$ where $L:=\min\{ \operatorname{level}^{\mathbb{R}f_\ast G} (\mathcal{O}_\mathcal{S}) : G \in D^b_{\operatorname{coh}}(\mathcal{Y}) \}$.
\end{corollary}

\begin{proof}
    Argue as in \cite[Proposition 3.16]{Lank/Olander:2024}, but use \Cref{cor:descendable_proper} and \Cref{lem:rouq_dim_stacks_via_approx}.
\end{proof}

\begin{example}\label{ex:rouq-tame-ms}
    Let $\mathcal{X}$ be an algebraic stack that is tame, Noetherian and has finite inertia. Denote the coarse moduli space by $\pi\colon \mathcal{X} \to \mathcal{X}_{\operatorname{cs}}$. Then $\mathcal{X}_{\operatorname{cs}}$ is a Noetherian algebraic space, $\pi$ is a concentrated and proper surjection and $\mathcal{O}_{\mathcal{X}_{\operatorname{cs}}} \to \mathbb{R}\pi_\ast \mathcal{O}_\mathcal{X}$ is an isomorphism in $D_{\operatorname{qc}}(\mathcal{X}_{\operatorname{cs}})$. By \Cref{cor:bound_rouquier_dimension}, $\dim D^b_{\operatorname{coh}}(\mathcal{X}_{\operatorname{cs}})\leq \dim D^b_{\operatorname{coh}}(\mathcal{X})$. 
\end{example}
\Cref{ex:rouq-tame-ms} essentially recovers \cite[Rem.~2.18]{Ballard/Favero:2012}. We can now prove \Cref{introthm:rouq_dim_bound_tame_stack_coarse_moduli}, which is similar %
to \cite[Lemma 2.17]{Ballard/Favero:2012}.
\begin{proof}[Proof of \Cref{introthm:rouq_dim_bound_tame_stack_coarse_moduli}]
    We are free to assume that $D^b_{\operatorname{coh}}(\mathcal{X})$ admits a strong generator as otherwise there is nothing to check. Let $G$ be an object of $D^b_{\operatorname{coh}}(\mathcal{X})$ and $n\geq 0$ such that $\langle G \rangle_n = D^b_{\operatorname{coh}}(\mathcal{X})$. By \cite[\href{https://stacks.math.columbia.edu/tag/08HP}{Tag 08HP}]{StacksProject} or \Cref{introthm:approximation_for_quasifinite}, $Y$ satisfies approximation by compacts. Hence, by \Cref{cor:bound_rouquier_dimension}, one has $D^b_{\operatorname{coh}}(Y) = \langle \mathbb{R}\pi_\ast G \rangle_n$. It suffices to show that $\dim Y \leq n$.
 
    Observe that the schematic locus of $Y$, denoted $U$, is a dense open subspace \cite[\href{https://stacks.math.columbia.edu/tag/06NH}{Tag 06NH}]{StacksProject}. Denote the associated open immersion by $j\colon U\to Y$. It suffices, from the Verdier localization $j^\ast \colon D^b_{\operatorname{coh}}(Y) \to D^b_{\operatorname{coh}}(U)$ (see \Cref{prop:verdier}), to show $\dim Y \leq \dim D^b_{\operatorname{coh}}(U)$. We can, if needed, intersect the regular locus of $Y$ with $U$ (which is nonempty as it contains generic point). Then, by shrinking, one can impose further $U$ is affine. Let $G$ be a strong generator for $D^b_{\operatorname{coh}}(U)$ with minimal generation time. There is, by \cite[Lemma 3.1]{Dey/Lank:2024}, a nonempty open subscheme $V$ of $U$ such that $G|_V$ and $\mathcal{O}_V$ finitely build one another in one cone. Our hypothesis on $Y$ ensures that $\dim V = \dim Y$
   Let $p$ be a closed point of $V$ such that $\dim \mathcal{O}_{V,p}=\dim V$. Denote the associated closed immersion by $i\colon \operatorname{Spec}(\kappa(p)) \to V$. It follows, from \cite[Corollary 4.3.13]{Letz:2020} coupled with \cite[\href{https://stacks.math.columbia.edu/tag/00OB}{Tag 00OB}]{StacksProject}, that $\operatorname{level}^{\mathcal{O}_{UV,p}}(i_\ast \mathcal{O}_{\kappa(p)}) = \dim X + 1$. Then the desired lower bound comes from the following string of inequalities:
   \begin{align*}
            \dim D^b_{\operatorname{coh}}(\mathcal{X}) &\geq \dim D^b_{\operatorname{coh}}(Y)
            \geq \dim D^b_{\operatorname{coh}}(U)
            = \operatorname{gen.time}(G) - 1\\
            &\geq \operatorname{gen.time}(G|_V) - 1
            = \operatorname{gen.time}(\mathcal{O}_V) -1
            \geq \operatorname{level}^{\mathcal{O}_{V,p}}(i_\ast \mathcal{O}_{\kappa(p)})  -1\\ 
            &= \dim V  
            = \dim Y. \qedhere
    \end{align*}
\end{proof}

\begin{remark} 
    We mention two variations of \Cref{introthm:rouq_dim_bound_tame_stack_coarse_moduli} for the interested reader. First, recall the \textbf{singularity category} of an algebraic stack $\mathcal{Y}$ is defined as the Verdier localization of $D^b_{\operatorname{coh}}(\mathcal{Y})/\operatorname{perf}(\mathcal{Y})$.\footnote{This was initially introduced in algebra by \cite{Buchweitz/Appendix:2021}, later rediscovered for schemes \cite{Orlov:2004}, and recently studied for algebraic stacks \cite{Bergh/Lunts/Schnurer:2016}.} Since the natural map $\mathcal{O}_Y \to \mathbb{R}f_\ast \mathcal{O}_{\mathcal{X}}$ splits, the projection formula implies that the unit map $E \to \mathbb{R}f_\ast \mathbb{L}f^\ast E$ splits. If $f$ is additionally assumed to be flat, then $\mathbb{R}f_\ast \colon D^b_{\operatorname{coh}}(\mathcal{X}) \to D^b_{\operatorname{coh}}(Y)$ and $\mathbb{R}f_\ast \colon \operatorname{Perf}(\mathcal{X}) \to \operatorname{Perf}(Y)$ are essentially dense. In particular, we obtain an (induced) exact functor $\overline{\pi}_\ast \colon D_{\operatorname{sg}}(\mathcal{X})\to D_{\operatorname{sg}}(Y)$, which is also essentially dense. 
    Also, one can show the `countable Rouquier dimension' (\`{a} la \cite{Olander:2023}) of $D^b_{\operatorname{coh}}(Y)$ is at most that of $D^b_{\operatorname{coh}}(\mathcal{X})$. See \cite{Peng:2024} for a recent study on countable Rouquier dimension for algebraic stacks of finite type over a field.
\end{remark}

\subsection{Ascent}
\label{sec:devissage_ascent}

This section studies the behavior of generation under derived pullback along a morphism of algebraic stacks. Our work significantly extends a previously known story regarding Noetherian schemes and finite Galois extensions, cf. \cite{Sosna:2014}. 
The following are instances which exhibit unfavorable behavior to generation in this context.

\begin{example}\label{ex:bad_behavior_1}
    Let $k$ be a perfect field. Choose an integer $n\geq 3$. Consider the field extension $k(x_1\ldots,x_n)$ over $k$ where each $x_i$ is transcendental over $k$. There exists a fibered square:
    \begin{displaymath}
        \begin{tikzcd}
            {X:=\operatorname{Spec}\big(k(x_1,\ldots,x_n)\otimes_k k(x_1,\ldots,x_n)\big)} & {\operatorname{Spec}\big(k(x_1,\ldots,x_n)\big)} \\
            {\operatorname{Spec}\big(k(x_1,\ldots,x_n)\big)} & {\operatorname{Spec}(k).}
            \arrow[from=1-1, to=1-2]
            \arrow[from=1-2, to=2-2]
            \arrow[from=1-1, to=2-1]
            \arrow[from=2-1, to=2-2]
        \end{tikzcd}
    \end{displaymath}
    Note that $X$ is an affine Noetherian scheme \cite[Theorem 11]{Vamos:1978}, has Krull dimension $n$ \cite{Sharp:1977}, and it is reduced \cite[$\S 5.15$, Thereom 3]{Bourbaki/AlgebraII:2003}. We know that the Rouquier dimension of $D^b_{\operatorname{coh}}(X)$ is at least $n-1$, see \cite[Corollary 6.6]{Aihara/Takahashi:2015}. However, the Rouquier dimension of $D^b_{\operatorname{coh}}(k(x_1,\ldots,x_n))$ is zero. 
\end{example}

\begin{example}\label{ex:bad_behavior_2}
    Let $k:= \mathbb{F}_3 (t)$ where $t$ is transcendental. Consider the projective plane curve $C$ over $k$ given by the equation $y^2 z + x^3 - t z^3 = 0$. This is a regular Noetherian scheme as it is normal. However, if we consider the base change $C^\prime$ of $C$ along the field extension $\mathbb{F}_3 (t^{\frac{1}{3}})$, then $C^\prime$ is a singular projective curve, see \cite[Remark 16]{Kollar:2011}. Let $\pi \colon C^\prime \to C$ be the projection morphism. We know that $C^\prime$ is an affine morphism. If $G$ is a compact generator for $D_{\operatorname{qc}}(C)$, then $\pi^\ast G$ is a compact generator for $D_{\operatorname{qc}}(C^\prime)$ \cite[\href{https://stacks.math.columbia.edu/tag/0BQT}{Tag 0BQT}]{StacksProject}. Note that any strong generator for $D^b_{\operatorname{coh}}(C)$ is a compact generator for $D_{\operatorname{qc}}(C)$ as $C$ is regular. Hence, $\pi^\ast G$ cannot be a strong generator for $D^b_{\operatorname{coh}}(C^\prime)$ as $C^\prime$ is singular. 
\end{example}

\Cref{ex:bad_behavior_1} tells us that Rouquier dimension can strictly increase along base of a field extension, whereas \Cref{ex:bad_behavior_2} is an instance where the derived pullback fails to preserve a strong generator if one base changes along a field extension. With these in mind, we attempt to identify a situation where such behavior ceases to exist, which motivates the following notion.

\begin{definition}\label{def:smooth_alteration}
    Let $\mathcal{B}$ be an algebraic stack. Let $\mathcal{Y}$ be an algebraic stack over $\mathcal{B}$. A \textbf{weak smooth $\mathcal{B}$-alteration} of $\mathcal{Y}$ is a concentrated, universally cohomologically proper, and surjective morphism $\pi \colon \tilde{\mathcal{Y}} \to \mathcal{Y}$ of algebraic stacks, where $\tilde{\mathcal{Y}}$ is smooth over $\mathcal{B}$.
\end{definition}

\begin{example}\label{ex:dejong-alteration}
  Let $k$ be a perfect field. Let $\mathcal{Y}$ be an algebraic stack of finite type over $k$ with quasi-finite and separated diagonal. Then there is a weak smooth $\operatorname{Spec}(k)$-alteration $\tilde{\mathcal{Y}} \to \mathcal{Y}$. Indeed, by \cite[Theorem B]{Rydh:2015}, there is a finite surjection $\mathcal{Y}_1 \to \mathcal{Y}$, where $\mathcal{Y}_1$ is a separated scheme of finite type over $k$. Now apply \cite{deJong-alterations} to $\mathcal{Y}_1$, which provides an alteration $\pi_1 \colon \tilde{\mathcal{Y}} \to \mathcal{Y}_1$, where $\pi_1$ is a proper and generically finite morphism of schemes and $\mathcal{Y}$ is regular. Since $k$ is perfect, $\mathcal{Y}$ is geometrically regular and so smooth over $k$; it follows that the composition $\tilde{\mathcal{Y}} \to \mathcal{Y}_1 \to \mathcal{Y}$ is a weak smooth $\operatorname{Spec}(k)$-alteration.
\end{example}

Before stating our result, we record a known lemma for convenience.

\begin{lemma}\label{lem:pullback_compact_generator_convservative}
    Suppose $\mathcal{T},\mathcal{S}$ are triangulated categories that are compactly generated by single objects. Let $\pi_\ast \colon \mathcal{T} \to \mathcal{S}$ be an exact functor which admits a left adjoint $\pi^\ast \colon \mathcal{S} \to \mathcal{T}$. Then the following are equivalent: 
    \begin{enumerate}
        \item $\pi_\ast \colon \mathcal{T} \to \mathcal{S}$ is conservative;\footnote{This means $\pi_\ast P\cong\pi_\ast Q$ implies $P\cong Q$.}
        \item $\pi_\ast \colon \mathcal{T} \to \mathcal{S}$ reflects the zero object;\footnote{This means $\pi_\ast E \cong 0$ implies $E\cong 0$.}
        \item $\pi^\ast P$ is a compact generator for $\mathcal{T}$ whenever $P$ is such for $\mathcal{S}$.
    \end{enumerate}
\end{lemma}

\begin{theorem}\label{thm:base_change_generation}
    Let $s \colon \mathcal{T} \to \mathcal{B}$ be a quasi-affine and flat morphism of algebraic stacks that are concentrated, regular and Noetherian. Let $\mathcal{Y}$ be a concentrated stack over $\mathcal{B}$ that admits a weak smooth $\mathcal{B}$-alteration $\widetilde{\mathcal{Y}} \to \mathcal{Y}$. Assume both $\mathcal{Y}$ and $\widetilde{\mathcal{Y}}$ satisfy the $1$-Thomason condition (e.g., have quasi-finite and separated diagonal). If $G$ is a classical generator for $D^b_{\operatorname{coh}}(\mathcal{Y})$, then $w^\ast G$ is a classical generator for $D^b_{\operatorname{coh}}(\mathcal{T} \times_\mathcal{B} \mathcal{Y})$, where $w\colon \mathcal{T} \times_\mathcal{B} \mathcal{Y}\to \mathcal{Y}$ is the natural projection morphism. Moreover, if $s$ is affine and $s_\ast \mathcal{O}_{\mathcal{T}} \simeq \mathcal{O}_{\mathcal{B}}^{\oplus I}$ for some set $I$ and $\langle w^\ast G \rangle_n = D^b_{\operatorname{coh}}(\mathcal{T} \times_{\mathcal{B}} \mathcal{Y})$, then $\langle G \rangle_n = D^b_{\operatorname{coh}}(\mathcal{Y})$.
\end{theorem}

\begin{proof}
    First, assume that $\mathcal{Y}$ is smooth over $\mathcal{B}$ (e.g.\ $\mathcal{Y}=\widetilde{\mathcal{Y}}$). Then $\mathcal{T} \times_{\mathcal{B}} \mathcal{Y} \to \mathcal{T}$ is also smooth and of finite presentation and so $\mathcal{Y}$ and $\mathcal{T} \times_{\mathcal{B}} \mathcal{Y}$ are Noetherian and regular. In particular, $D^b_{\operatorname{coh}}(\mathcal{Y})=\operatorname{Perf}(\mathcal{Y})$ and $D^b_{\operatorname{coh}}(\mathcal{T} \times_{\mathcal{B}} \mathcal{Y}) = \operatorname{Perf}(\mathcal{T} \times_{\mathcal{B}} \mathcal{Y})$ (\Cref{ex:strong_generator_examples}). By \cite[Lemma 2.5(1)]{Hall/Rydh:2017}, $\mathcal{T} \times_{\mathcal{B}} \mathcal{Y}$ is concentrated. Since $w$ is quasi-affine, $\mathbb{R}w_\ast \colon D_{\operatorname{qc}} (\mathcal{T} \times_{\mathcal{B}} \mathcal{Y}) \to D_{\operatorname{qc}}(\mathcal{Y})$ is conservative (\cite[Corollary 2.8]{Hall/Rydh:2017}). Hence, \Cref{lem:pullback_compact_generator_convservative} tells us that $w^\ast G$ is a compact generator for $D_{\operatorname{qc}}(\mathcal{T} \times_{\mathcal{B}} \mathcal{Y})$, and so, $w^\ast G$ is a classical generator for $D^b_{\operatorname{coh}}(\mathcal{T} \times_{\mathcal{B}} \mathcal{Y})$ as desired.
    
    In general, form the fibered square: %
    \begin{displaymath}
        \begin{tikzcd}
            {\mathcal{T}\times_\mathcal{B} \widetilde{\mathcal{Y}}} & {\widetilde{\mathcal{Y}}} \\
            {\mathcal{T}\times_\mathcal{B} \mathcal{Y}} & \mathcal{Y}.
            \arrow["{\tilde{w}}", from=1-1, to=1-2]
            \arrow["\pi", from=1-2, to=2-2]
            \arrow["{1_\mathcal{T} \times_\mathcal{B} \pi}"', from=1-1, to=2-1]
            \arrow["w"', from=2-1, to=2-2]
            \arrow["\lrcorner"{anchor=center, pos=0.125}, draw=none, from=1-1, to=2-2]
        \end{tikzcd}
    \end{displaymath}
    By base change, $\tilde{w}$ is flat and quasi-affine. 
    Note, by \cite[Lemma 8.2]{Hall/Rydh:2017}, that both $\mathcal{T} \times_\mathcal{B} \widetilde{\mathcal{Y}}$ and $\mathcal{T}\times_\mathcal{B} \mathcal{Y}$ also satisfy the $1$-Thomason condition.
    Then, from \Cref{cor:descendable_proper}, one has that $\mathbb{R}\pi_\ast G$ is a classical generator for $D^b_{\operatorname{coh}}(\mathcal{Y})$ if $G$ is a compact generator for $D_{\operatorname{Qcoh}}(\widetilde{\mathcal{Y}})$. It suffices to check that $w^\ast \mathbb{R}\pi_\ast G$ is a classical generator for $D^b_{\operatorname{coh}}(\mathcal{T} \times_{\mathcal{B}} \mathcal{Y})$. However, we can use flat base change to verify this, see \cite[Lemma 1.2(4)]{Hall/Rydh:2017}. As noted previously, $\tilde{w}^\ast G$ is a classical generator for $D^b_{\operatorname{coh}}(\mathcal{T} \times_\mathcal{B} \widetilde{\mathcal{Y}})$. Since the morphism $1_{\mathcal{T}}\times_{\mathcal{B}} \pi$ is universally cohomologically proper, concentrated, and surjective, \Cref{cor:descendable_proper} tells us that $\mathbb{R}(1_\mathcal{T} \times_\mathcal{B} \pi )_\ast \tilde{w}^\ast G$ is a classical generator for $D^b_{\operatorname{coh}}(\mathcal{Y}\times_\mathcal{B} \mathcal{T})$. However, we know that $\mathbb{R}(1_\mathcal{T} \times_\mathcal{B} \pi )_\ast \tilde{w}^\ast G$ is isomorphic to $w^\ast \mathbb{R}\pi_\ast G$ in $D^b_{\operatorname{coh}}(\mathcal{Y}\times_\mathcal{B} \mathcal{T})$, which furnishes the claim.
  
   For the second part: let $E$ be an object of $D^b_{\operatorname{coh}} (\mathcal{Y})$. %
    If $y \colon \mathcal{Y} \to \mathcal{B}$ is the structure morphism, then
    \begin{displaymath}
        \begin{aligned}
            \mathbb{R}w_\ast w^\ast E &\cong w_\ast \mathcal{O}_{\mathcal{T} \times_{\mathcal{B}} \mathcal{Y}} \otimes^{\mathbb{L}} E && \textrm{(projection formula)} 
            \\&\cong w_\ast (1_{\mathcal{T}} \times_{\mathcal{B}} y) ^\ast\mathcal{O}_{\mathcal{T}} \otimes^{\mathbb{L}} E && \\
            &\cong \pi^\ast s_\ast \mathcal{O}_{\mathcal{T}} \otimes^{\mathbb{L}} E && \textrm{(flat base base change)}
            \\&\cong E^{\oplus I}. && %
        \end{aligned}
    \end{displaymath}
    If $\pi^\ast G$ is a strong generator for $D^b_{\operatorname{coh}}(\mathcal{T}\times_{\mathcal{B}}\mathcal{Y})$ with generation time $N$, then $\pi^\ast E$ belongs to $\langle \pi^\ast G \rangle_{N+1}$. Then $\pi_\ast \pi^\ast E$ is in $\overline{\langle \pi_\ast \pi^\ast G \rangle}_{N+1}$, and yet $\pi_\ast \pi^\ast G$ belongs to $\overline{\langle G \rangle}_1$. This tells us $E$ is in $\overline{\langle G \rangle}_{N+1}$, and so $E$ belongs to $\langle G \rangle_{N+1}$ \cite[Lemma 2.14]{DeDeyn/Lank/ManaliRahul:2024}. This completes the proof.
\end{proof}

\begin{proof}[Proof of \Cref{introthm:base_change_generation}]
This follows from  \Cref{thm:base_change_generation}, by taking $\mathcal{B} = \operatorname{Spec}(k)$, $\mathcal{T} = \operatorname{Spec}(\ell)$ and $\tilde{\mathcal{Y}} \to \mathcal{Y}$ as in \Cref{ex:dejong-alteration}.
\end{proof}

\begin{appendix}

\section{Some Deligne style formulas for algebraic stacks}
The following lemma we expect to be well-known, though we were unable to find a precise reference in the literature. It is closely related to a formula commonly attributed to Deligne.

\begin{lemma}\label{lem:noetherian-deligne-pushforward}
    Let $\mathcal{X}$ be a locally Noetherian algebraic stack, $j \colon \mathcal{U} \hookrightarrow \mathcal{X}$ be a quasi-compact open immersion, and $\mathcal{I}$ be a coherent sheaf of ideals in $\mathcal{O}_{\mathcal{X}}$ with $\mathcal{U}=\mathcal{X} \setminus V(\mathcal{I})$. If $N$ is in $D_{\operatorname{qc}}^+(\mathcal{X})$, then the canonical isomorphism $\mathbb{L}j^\ast \mathbb{R}\operatorname{\mathcal{H}\! \mathit{om}}_{\mathcal{O}_{\mathcal{X}}}(\mathcal{I}^n,N)$ with $\mathbb{L}j^\ast N$ induces an isomorphism in $D_{\operatorname{qc}}(\mathcal{X})$:
    \begin{displaymath}
        \phi_N \colon \operatorname{hocolim}_{n} \mathbb{R}\operatorname{\mathcal{H}\! \mathit{om}}_{\mathcal{O}_{\mathcal{X}}}(\mathcal{I}^n,N) \to \mathbb{R} j_\ast \mathbb{L} j^\ast N. 
    \end{displaymath}
\end{lemma}

\begin{proof}
    The statement is local on $\mathcal{X}$ for the smooth topology, so we may assume that $\mathcal{X}=\operatorname{Spec}(A)$, where $A$ is a Noetherian ring. Now form the distinguished triangle:
    \begin{displaymath}
        \mathcal{C}_N \to \operatorname{hocolim}_{n} \mathbb{R}\operatorname{\mathcal{H}\! \mathit{om}}_{\mathcal{O}_{\mathcal{X}}}(\mathcal{I}^n,N) \to \mathbb{R} j_\ast \mathbb{L} j^\ast N \to \mathcal{C}_N[1].
    \end{displaymath}    
    There is an induced morphism of distinguished triangles:
    \begin{displaymath}
        \begin{tikzcd}[ampersand replacement=\&, column sep=2ex]
            {\mathcal{C}_{N}} \& {\operatorname{hocolim}_{n} \mathbb{R}\operatorname{\mathcal{H}\! \mathit{om}}_{\mathcal{O}_{\mathcal{X}}}(\mathcal{I}^n,N)} \& { \mathbb{R} j_\ast \mathbb{L} j^\ast N} \& { \mathcal{C}_{N}[1]} \\
            {\mathcal{C}_{\mathbb{R}j_\ast \mathbb{L} j^\ast N} } \& {Q} \& {\mathbb{R} j_\ast \mathbb{L} j^\ast \mathbb{R}j_\ast \mathbb{L} j^\ast N} \& {\mathcal{C}_{\mathbb{R}j_\ast \mathbb{L} j^\ast N}[1] }
            \arrow[from=1-1, to=1-2]
            \arrow[from=1-1, to=2-1]
            \arrow[from=1-2, to=1-3]
            \arrow[from=1-2, to=2-2]
            \arrow[from=1-3, to=1-4]
            \arrow["\cong", from=1-3, to=2-3]
            \arrow[from=1-4, to=2-4]
            \arrow[from=2-1, to=2-2]
            \arrow[from=2-2, to=2-3]
            \arrow[from=2-3, to=2-4]
        \end{tikzcd}
    \end{displaymath}
    where $Q:= \operatorname{hocolim}_{n} \mathbb{R}\operatorname{\mathcal{H}\! \mathit{om}}_{\mathcal{O}_{\mathcal{X}}}(\mathcal{I}^n,\mathbb{R}j_\ast \mathbb{L} j^\ast N)$. Clearly, $\phi_{\mathbb{R} j_\ast \mathbb{L}j^\ast N}$ is an isomorphism. Thus, it suffices to prove that $\operatorname{hocolim}_{n} \mathbb{R}\operatorname{\mathcal{H}\! \mathit{om}}_{\mathcal{O}_{\mathcal{X}}}(\mathcal{I}^n,N)$ is zero if $\mathbb{L} j^\ast N \simeq 0$. Equivalently, $\operatorname{hocolim}_{n} \mathbb{R}\operatorname{\mathcal{H}\! \mathit{om}}_{\mathcal{O}_{\mathcal{X}}}(\mathcal{O}_{\mathcal{X}}/\mathcal{I}^n, N) \to N$ is an isomorphism whenever $\mathbb{L} j^\ast N \simeq 0$, which is just \cite[\href{https://stacks.math.columbia.edu/tag/0954}{Tag 0954} \& \href{https://stacks.math.columbia.edu/tag/0955}{Tag 0955}]{StacksProject}.
\end{proof}

The next proposition gives us a method to compute local cohomology in non-affine and non-Noetherian situations. While less striking than what is available in the affine case (cf.\ \cite[\href{https://stacks.math.columbia.edu/tag/0DWQ}{Tag 0DWQ}]{StacksProject}), it is sufficient for our purposes.

\begin{proposition}\label{prop:nonnoetherian-deligne-formula}
    Let $\mathcal{X}$ be a quasi-compact and quasi-separated algebraic stack. If $j\colon \mathcal{U} \hookrightarrow \mathcal{X}$ is a quasi-compact open immersion, then there exists an inverse system of pseudocoherent complexes in $D_{\operatorname{qc}}^{\leq 0}(\mathcal{X})$:
    \begin{displaymath}
        \mathcal{O}_{\mathcal{X}}=\mathcal{I}_0 \xleftarrow{\iota_0} \mathcal{I}_1 \xleftarrow{\iota_1} \mathcal{I}_2 \xleftarrow{\iota_2}\cdots 
    \end{displaymath}
    such that the $(\iota_i)_{\mathcal{U}}$ are isomorphisms and if $N \in D_{\operatorname{qc}}^+(\mathcal{X})$, then the isomorphisms $\mathbb{L}j^\ast \mathbb{R}\operatorname{\mathcal{H}\! \mathit{om}}_{\mathcal{O}_{\mathcal{X}}}(\mathcal{I}_n,N) \cong \mathbb{L} j^\ast N$ induce by adjunction an isomorphism 
    \begin{displaymath}
        \psi_{N} \colon \operatorname{hocolim}_{n} \mathbb{R}\operatorname{\mathcal{H}\! \mathit{om}}_{\mathcal{O}_{\mathcal{X}}}(\mathcal{I}_n,N) \cong \mathbb{R}j_\ast  \mathbb{L}j^\ast N. 
    \end{displaymath}
    In particular, if $P\in D_{\operatorname{qc}}^-(\mathcal{X})$ is pseudocoherent, then there is an induced isomorphism:
    \begin{displaymath}
        \mathbb{R}\operatorname{Hom}_{\mathcal{O}_{\mathcal{U}}}(\mathbb{L}j^\ast P,\mathbb{L} j^\ast \mathcal{N)} \cong \operatorname{hocolim}_{n}   \mathbb{R}\operatorname{Hom}_{\mathcal{O}_{\mathcal{X}}}(\mathcal{I}_n \otimes^{\mathbb{L}}_{\mathcal{O}_{\mathcal{X}}} P, N).
    \end{displaymath}
    Further, for each $n$ let $\tilde{\iota}_n = \iota_n \circ \iota_{n-1} \circ \cdots \iota_0$ and form the distinguished triangle:
    \begin{displaymath}
        \mathcal{I}_n \xrightarrow{\tilde{\iota}_n} \mathcal{O}_{\mathcal{X}} \xrightarrow{\tilde{\kappa}_n} \mathcal{K}_n \to \mathcal{I}_n[1].
    \end{displaymath}
    Then there is an induced inverse system of pseudocoherent complexes in $D^{\leq 0}_{\operatorname{qc}}(\mathcal{X})$:
    \begin{displaymath}
        0 = \mathcal{K}_0 \xleftarrow{\kappa_0} \mathcal{K}_1 \xleftarrow{\kappa_1} \mathcal{K}_2 \xleftarrow{\kappa_2} \cdots
    \end{displaymath}
    such that $(\mathcal{K}_i)_{\mathcal{U}} = 0$.  If $N \in D_{\operatorname{qc}}^+(\mathcal{X})$, then there is a induced distinguished triangle:
    \begin{displaymath}
        \mathbb{R} j_\ast \mathbb{L} j^\ast N[-1] \to \operatorname{hocolim}_{n} \mathbb{R}\operatorname{\mathcal{H}\! \mathit{om}}_{\mathcal{O}_{\mathcal{X}}}(\mathcal{K}_n,N) \to N \to \mathbb{R} j_\ast \mathbb{L} j^\ast N.
    \end{displaymath}
    In particular, if $P \in D_{\operatorname{qc}}^-(\mathcal{X})$ is pseudocoherent and $\phi \colon P \to N$ is a morphism such $j^\ast \phi = 0$, then there exists a factorization $P \xrightarrow{P \otimes \tilde{\kappa}_n }\mathcal{K}_n \otimes^{\mathbb{L}}_{\mathcal{O}_{\mathcal{X}}} P \xrightarrow{\phi_n} N$. Finally, if $\mathcal{X}$ has the resolution property, then for each $n$ there is a morphism $\chi_n \colon \mathcal{Q}_n \to \mathcal{K}_n$, where $\mathcal{Q}_n \in D^{\leq 0}_{\operatorname{qc}}(\mathcal{X})$ is perfect, $(\mathcal{Q}_n)_{\mathcal{U}} \cong 0$ and $H^0(\chi_n)$ is an isomorphism.
\end{proposition}

\begin{proof}
    Let $\alpha \colon \mathcal{X} \to \mathcal{X}_0$ be an absolute Noetherian approximation of $\mathcal{X}$ \cite{Rydh:2023} such that that there is an open immersion $j_0 \colon \mathcal{U}_0 \hookrightarrow \mathcal{X}_0$ such that $\alpha^{-1}(\mathcal{U}_0) = \mathcal{U}$ as open substacks of $\mathcal{X}$. Let $\mathcal{I} \subseteq \mathcal{O}_{\mathcal{X}_0}$ be a coherent sheaf of ideals with $\mathcal{U}_0 = \mathcal{X} - V(\mathcal{I})$. Let $\mathcal{I}_n = \mathbb{L} \alpha^\ast  \mathcal{I}^n$; then $\mathcal{I}_0^{n+1} \subseteq \mathcal{I}_0^n$ induces a morphism $\mathcal{I}_n \xleftarrow{\iota_n} \mathcal{I}_{n+1}$ such that $(\iota_n)_{\mathcal{U}}$ is an isomorphism and the $\mathcal{I}_n$ are all pseudo-coherent complexes in $D^{\leq 0}_{\operatorname{qc}}(\mathcal{X})$. Since $\alpha$ is affine, $\mathbb{R} \alpha_\ast  \colon D_{\operatorname{qc}}(\mathcal{X}) \to D_{\operatorname{qc}}(\mathcal{X}_0)$ is conservative \cite[Corollary 2.8]{Hall/Rydh:2017}, so it suffices to prove that $\mathbb{R} \alpha_\ast  \psi_{N}$ is an isomorphism. But
    this map factors through the following composition of isomorphisms:
    \begin{displaymath}
        \begin{aligned}
            \mathbb{R} \alpha_\ast & \operatorname{hocolim}_{n} \mathbb{R}\operatorname{\mathcal{H}\! \mathit{om}}_{\mathcal{O}_{\mathcal{X}}}(\mathcal{I}_n,N) 
            \\&\cong \operatorname{hocolim}_{n} \mathbb{R} \alpha_\ast  \mathbb{R}\operatorname{\mathcal{H}\! \mathit{om}}_{\mathcal{O}_{\mathcal{X}}}(\mathbb{L} \alpha^\ast \mathcal{I}^n,N) && \mbox{\cite[Theorem 2.6]{Hall/Rydh:2017}} \\
            &\cong \operatorname{hocolim}_{n} \mathbb{R}\operatorname{\mathcal{H}\! \mathit{om}}_{\mathcal{O}_{\mathcal{X}_0}}(\mathcal{I}^n,\mathbb{R} \alpha_\ast N)\\
            &\cong \mathbb{R} (j_0)_\ast  \mathbb{R} j_0^\ast \mathbb{R} \alpha_\ast  N && \mbox{(\Cref{lem:noetherian-deligne-pushforward})}\\
            &\cong \mathbb{R} (j_0)_\ast  \mathbb{R} \alpha_{\mathcal{U},*} \mathbb{L} j^\ast N && \mbox{\cite[Corollary 4.13]{Hall/Rydh:2017}} \\
            &\cong \mathbb{R} \alpha_\ast  \mathbb{R} j_\ast  \mathbb{L} j^\ast N,
        \end{aligned}
    \end{displaymath}
    which gives the claim. Finally, since $N \in D^{\geq -r}_{\operatorname{qc}}(\mathcal{X})$ for some $r$ and $\mathcal{I}_n \in D^{\leq 0}_{\operatorname{qc}}(\mathcal{X})$ is pseudocoherent, $\mathbb{R} \operatorname{\mathcal{H}\! \mathit{om}}_{\mathcal{O}_{\mathcal{X}}}(\mathcal{I}_n,N) \in D^{\geq -r}_{\operatorname{qc}}(\mathcal{X})$ for all $n$. Since $P$ is pseudocoherent, it follows from the above that:
    \begin{displaymath}
        \begin{aligned}
            \mathbb{R}\operatorname{Hom}_{\mathcal{O}_{\mathcal{U}}} & (\mathbb{L}j^\ast P,\mathbb{L} j^\ast N) 
            \\&\cong  \mathbb{R}\operatorname{Hom}_{\mathcal{O}_{\mathcal{X}}}(P,\mathbb{R}j_\ast \mathbb{L} j^\ast N)\\
            &\cong \mathbb{R}\operatorname{Hom}_{\mathcal{O}_{\mathcal{X}}}(P,\operatorname{hocolim}_{n} \mathbb{R}\operatorname{\mathcal{H}\! \mathit{om}}_{\mathcal{O}_{\mathcal{X}}}(\mathcal{I}_n,N))\\
            &\cong \operatorname{hocolim}_{n}\mathbb{R}\operatorname{Hom}_{\mathcal{O}_{\mathcal{X}}}(P, \mathbb{R}\operatorname{\mathcal{H}\! \mathit{om}}_{\mathcal{O}_{\mathcal{X}}}(\mathcal{I}_n,N))\\
            &\cong \operatorname{hocolim}_{n}   \mathbb{R}\operatorname{Hom}_{\mathcal{O}_{\mathcal{X}}}(\mathcal{I}_n \otimes^{\mathbb{L}}_{\mathcal{O}_{\mathcal{X}}} P, N).
        \end{aligned}
    \end{displaymath}
    Most of the latter claims follow easily, except the last
    one. Since $\mathcal{X}$ has the resolution property, it can be
    arranged that $\mathcal{X}_0$ has the resolution property. Then
    there is a vector bundle $E_0$ on $\mathcal{X}_0$ and a surjection
    $E_0 \to \mathcal{I}^n$, which corresponds to a homomorphism
    $s^\vee \colon E_0 \to \mathcal{O}_{\mathcal{X_0}}$. By
    \cite[V.2]{Fulton/Lang:1985}, there is an associated Koszul
    complex $K(s^\vee)$ on $\mathcal{X}_0$ that comes with a morphism
    $\chi_n^0 \colon K(s^\vee) \to
    \mathcal{O}_{\mathcal{X}_0}/\mathcal{I}^n$, where
    $K(s^\vee) \in D^{\leq 0}_{\operatorname{qc}}(\mathcal{X}_0)$ is
    perfect, $K(s^\vee)_{\mathcal{U}_0} \cong 0$ and $H^0(\chi_n^0)$
    is an isomorphism. Taking $\chi_n = \mathbb{L}\alpha^\ast\chi_n^0$
    does the job.
\end{proof}

\section{Verdier localizations for $D^b_{\operatorname{coh}}$ of algebraic stacks}
\label{sec:verdier_localization_algebraic_stacks}

We prove a statement regarding Verdier localizations for objects with bounded and coherent cohomology on a Noetherian algebraic stack. If $\mathcal{X}$ is Noetherian and affine-pointed, it follows by combining \cite[Theorem C.1]{Hall/Neeman/Rydh:2019} and \cite[Theorem 4.4]{Elagin/Lunts/Schnurer:2020}. \vspace{-0.5cm}
\begin{proposition}\label{prop:verdier}
    Let $\mathcal{X}$ be a Noetherian algebraic stack. If $j\colon \mathcal{U}\to \mathcal{X}$ is an open immersion, then there exists a Verdier localization\:
    \begin{displaymath}
        D^b_{\operatorname{coh},|\mathcal{X}| \setminus |\mathcal{U}|} (\mathcal{X}) \to D^b_{\operatorname{coh}}(\mathcal{X}) \xrightarrow{j^\ast} D^b_{\operatorname{coh}}(\mathcal{U}).
    \end{displaymath}
\end{proposition}

\begin{proof}
    If $\mathcal{W}$ is a Noetherian algebraic stack, let $\overline{\operatorname{coh}}(\mathcal{W}) \subset \operatorname{Qcoh}(\mathcal{W})$ denote the full subcategory with objects those quasi-coherent $\mathcal{O}_{\mathcal{W}}$-modules $M$ such that $\Gamma(\operatorname{Spec}(R),M)$ is a countably generated $R$-module for all smooth morphisms $\operatorname{Spec}(R) \to \mathcal{W}$. By \cite[Lem.~A.1]{Hall:2022},  $\overline{\operatorname{coh}}(\mathcal{W})$ is a Serre subcategory of $\operatorname{Qcoh}(\mathcal{W})$. In particular, the full subcategory $D^b_{\overline{\operatorname{coh}}}(\mathcal{W}) \subseteq D_{\operatorname{qc}}(\mathcal{W})$ bounded complexes of $\operatorname{O}_{\mathcal{W}}$-modules with cohomology groups in $\overline{\operatorname{coh}}(\mathcal{W})$ is a triangulated subcategory. 
    
    Let $i \colon \mathcal{Z} \hookrightarrow \mathcal{X}$ be a closed immersion with $|\mathcal{Z}|^c = |\mathcal{U}|$ as subsets of $|\mathcal{X}|$. Note that if $M \in D^b_{\overline{\operatorname{coh}}}(\mathcal{U})$, then $\mathbb{R}j_\ast M \in D^b_{\overline{\operatorname{coh}}}(\mathcal{X})$. Indeed, this is local on $\mathcal{X}$ for the smooth topology, so we may assume that $\mathcal{X} = \operatorname{Spec}(R)$ is a Noetherian and affine scheme, where it now follows from a simple argument involving the C\v{e}ch complex (cf.\ \cite[Lem.~A.4]{Hall:2022}). By \cite[Lem.~3.4]{Hall/Rydh:2017}, it follows that there is a Verdier localization:
    \begin{displaymath}
        D^b_{\overline{\operatorname{coh}},|\mathcal{Z}|} (\mathcal{X}) \to D^b_{\overline{\operatorname{coh}}}(\mathcal{X}) \xrightarrow{j^\ast} D^b_{\overline{\operatorname{coh}}}(\mathcal{U}).
    \end{displaymath}
    Noting that $D^b_{\operatorname{coh},|\mathcal{Z}|} (\mathcal{X}) = D^b_{\overline{\operatorname{coh}},|\mathcal{Z}|} (\mathcal{X}) \cap D^b_{\operatorname{coh}} (\mathcal{X})$ as subcategories of $D^+_{\operatorname{qc}}(\mathcal{X})$, the definition of a Verdier localization shows it suffices to prove the following two claims:
    \begin{claim}\label{claim:verdier-surj}
        If $M \in D^b_{\operatorname{coh}}(\mathcal{U})$, then there exists $\overline{M} \in D^b_{\operatorname{coh}}(\mathcal{X})$ such that $j^\ast\overline{M} \simeq M$ in $D^b_{\operatorname{coh}}(\mathcal{U})$.
    \end{claim}
    \begin{claim}\label{claim:verdier-faith}
        If $\phi \colon E \to F$ is a morphism in $D^b_{\overline{\operatorname{coh}}}(\mathcal{X})$ with $j^\ast \phi$ an isomorphism and $F \in D^b_{\operatorname{coh}}(\mathcal{X})$, then there exists a map $\psi \colon E' \to E$ in $D^b_{\overline{\operatorname{coh}}}(\mathcal{X})$ with $j^\ast \psi$ an isomorphism and $E' \in D^b_{\operatorname{coh}}(\mathcal{X})$.
    \end{claim}
    For \Cref{claim:verdier-surj}, we note that $\mathbb{R}j_\ast M$ belongs to $D^b_{\overline{\operatorname{coh}}}(\mathcal{X})$ and so is quasi-isomorphic to a homotopy colimit $\operatorname{hocolim}_{s} M_s$, where $M_s \in D^b_{\operatorname{coh}}(\mathcal{X})$ and the induced maps $\mathcal{H}^k(M_s) \to \mathcal{H}^k(\mathbb{R} j_\ast M)$ are monomorphisms for all $k \in \mathbf{Z}$ \cite[Lem.~A.3]{Hall:2022}. Since $j$ is an open immersion, it is flat and so the induced maps $j^\ast M_s \to j^\ast \mathbb{R} j_\ast M \simeq M$ induce monomorphisms $\mathcal{H}^k(j^\ast M_s) \to \mathcal{H}^k(M)$ for all $k\in \mathbf{Z}$. Since $M \in D^b_{\operatorname{coh}}(\mathcal{U})$ it follows that we may take $s \gg 0$ such that $j^\ast M_s \simeq M$, which gives Claim \ref{claim:verdier-surj}. For Claim \ref{claim:verdier-faith}, form the distinguished triangle:
    \[
        G \to E \xrightarrow{\phi} F \to G[1].
    \]
    Then $j^\ast G \simeq 0$ and so $G \in D^b_{\overline{\operatorname{coh}},|\mathcal{Z}|}(\mathcal{X})$. We apply \cite[Lem.~A.3]{Hall:2022} again to write $G \simeq \operatorname{hocolim}_s G_s$, where $G_s \in D^b_{\operatorname{coh}}(\mathcal{X})$ and $\mathcal{H}^k(G_s) \to \mathcal{H}^k(G)$ a monomorphism for all $k\in \mathbf{Z}$. It follows that $j^\ast G_s \simeq 0$ for all $s$ and so $G_s \in D^b_{\operatorname{coh},|\mathcal{Z}|}(\mathcal{X})$ as well. Since $F$ belongs to $D^b_{\operatorname{coh}}(\mathcal{X})$, it follows from \cite[Lem.~1.2]{Hall/Rydh:2017} that $F \to G[1]$ factors through $G_s[1] \to G[1]$ for some $s \gg 0$. There is an induced morphism of distinguished triangles
    \[
    \xymatrix{G_s \ar[d] \ar[r] & E' \ar[d] \ar[r] & F \ar@{=}[d]\ar[r] & G_s[1] \ar[d] \\ G \ar[r] & E \ar[r] & F \ar[r] & G[1],}
    \]
    which gives \Cref{claim:verdier-faith}.
\end{proof}

\end{appendix}

\bibliographystyle{alpha}
\bibliography{mainbib}

\newcommand{\etalchar}[1]{$^{#1}$}
\begin{thebibliography}{AHHLR24}

\bibitem[ABIM10]{ABIM:2010}
Luchezar~L. Avramov, Ragnar-Olaf Buchweitz, Srikanth~B. Iyengar, and Claudia Miller.
\newblock Homology of perfect complexes.
\newblock {\em Adv. Math.}, 223(5):1731--1781, 2010.

\bibitem[AHHLR24]{Alper/Hall/Halpern-Leistner/Rydh:2024}
Jarod Alper, Jack Hall, Daniel Halpern-Leistner, and David Rydh.
\newblock Artin algebraization for pairs with applications to the local structure of stacks and {F}errand pushouts.
\newblock {\em Forum Math. Sigma}, 12:Paper No. e20, 25, 2024.

\bibitem[AHL23]{Alper/Hall/BenjaminLim:2023}
Jarod Alper, Jack Hall, and David~Benjamin Lim.
\newblock Coherently complete algebraic stacks in positive characteristic.
\newblock \href{https://arxiv.org/abs/2309.01388}{arXiv:2309.01388}, 2023.

\bibitem[Aok21]{Aoki:2021}
Ko~Aoki.
\newblock Quasiexcellence implies strong generation.
\newblock {\em J. Reine Angew. Math.}, 780:133--138, 2021.

\bibitem[AT15]{Aihara/Takahashi:2015}
Takuma Aihara and Ryo Takahashi.
\newblock Generators and dimensions of derived categories of modules.
\newblock {\em Commun. Algebra}, 43(11):5003--5029, 2015.

\bibitem[Bal16]{Balmer:2016}
Paul Balmer.
\newblock Separable extensions in tensor-triangular geometry and generalized {Quillen} stratification.
\newblock {\em Ann. Sci. {\'E}c. Norm. Sup{\'e}r. (4)}, 49(4):907--925, 2016.

\bibitem[BF12]{Ballard/Favero:2012}
Matthew Ballard and David Favero.
\newblock Hochschild dimensions of tilting objects.
\newblock {\em Int. Math. Res. Not.}, 2012(11):2607--2645, 2012.

\bibitem[BIL{\etalchar{+}}23]{BILMP:2023}
Matthew~R. Ballard, Srikanth~B. Iyengar, Pat Lank, Alapan Mukhopadhyay, and Josh Pollitz.
\newblock High frobenius pushforwards generate the bounded derived category.
\newblock \href{https://arxiv.org/abs/2303.18085}{arXiv:2303.18085}, 2023.

\bibitem[BLS16]{Bergh/Lunts/Schnurer:2016}
Daniel Bergh, Valery~A. Lunts, and Olaf~M. Schn{\"u}rer.
\newblock Geometricity for derived categories of algebraic stacks.
\newblock {\em Sel. Math., New Ser.}, 22(4):2535--2568, 2016.

\bibitem[Bou03]{Bourbaki/AlgebraII:2003}
Nicolas Bourbaki.
\newblock {\em Elements of mathematics. {Algebra} {II}. {Chapters} 4--7. {Transl}. from the {French} by {P}. {M}. {Cohn} and {J}. {Howie}.}
\newblock Berlin: Springer, reprint of the 1990 {English} translation edition, 2003.

\bibitem[BS17]{Bhatt/Scholze:2017}
Bhargav Bhatt and Peter Scholze.
\newblock Projectivity of the {W}itt vector affine {G}rassmannian.
\newblock {\em Invent. Math.}, 209(2):329--423, 2017.

\bibitem[Buc21]{Buchweitz/Appendix:2021}
Ragnar-Olaf Buchweitz.
\newblock {\em Maximal {Cohen}-{Macaulay} modules and {Tate} cohomology. {With} appendices by {Luchezar} {L}. {Avramov}, {Benjamin} {Briggs}, {Srikanth} {B}. {Iyengar} and {Janina} {C}. {Letz}}, volume 262 of {\em Math. Surv. Monogr.}
\newblock Providence, RI: American Mathematical Society (AMS), 2021.

\bibitem[BvdB03]{BVdB:2003}
A.~Bondal and M.~van~den Bergh.
\newblock Generators and representability of functors in commutative and noncommutative geometry.
\newblock {\em Mosc. Math. J.}, 3(1):1--36, 2003.

\bibitem[dJ96]{deJong-alterations}
Aise~Johan de~Jong.
\newblock Smoothness, semi-stability and alterations.
\newblock {\em Inst. Hautes \'{E}tudes Sci. Publ. Math.}, (83):51--93, 1996.

\bibitem[DL24a]{Dey/Lank:2024a}
Souvik Dey and Pat Lank.
\newblock Closedness of the singular locus and generation for derived categories.
\newblock \href{https://arxiv.org/abs/2403.19564}{arXiv:2403.19564}, 2024.

\bibitem[DL24b]{Dey/Lank:2024}
Souvik Dey and Pat Lank.
\newblock D\'{e}vissage for generation in derived categories.
\newblock \href{https://arxiv.org/abs/2401.13661}{arXiv:2401.13661}, 2024.

\bibitem[DLM24]{DeDeyn/Lank/ManaliRahul:2024}
Timothy~De Deyn, Pat Lank, and Kabeer {Manali Rahul}.
\newblock Approximability and rouquier dimension for noncommuative algebras over schemes.
\newblock \href{https://arxiv.org/abs/2408.04561}{arXiv:2408.04561}, 2024.

\bibitem[DLM25]{DeDeyn/Lank/KabeerManaliRahul:2025}
Timothy {De Deyn}, Pat Lank, and Kabeer {Manali Rahul}.
\newblock Descending strong generation in algebraic geometry.
\newblock \href{https://arxiv.org/abs/2502.08629}{arXiv:2502.08629}, 2025.

\bibitem[DLMP25]{DeDeyn/Lank/KabeerManaliRahul/Peng:2025}
Timothy {De Deyn}, Pat Lank, Kabeer {Manali Rahul}, and Fei Peng.
\newblock Regularity and bounded $t$-structures for algebraic stacks, 2025.

\bibitem[DLT23]{Dey/Lank/Takahashi:2023}
Souvik Dey, Pat Lank, and Ryo Takahashi.
\newblock Strong generation for module categories.
\newblock \href{https://arxiv.org/abs/2307.13675}{arXiv:2307.13675}, 2023.

\bibitem[ELS20]{Elagin/Lunts/Schnurer:2020}
Alexey Elagin, Valery~A. Lunts, and Olaf~M. Schn{\"u}rer.
\newblock Smoothness of derived categories of algebras.
\newblock {\em Mosc. Math. J.}, 20(2):277--309, 2020.

\bibitem[FL85]{Fulton/Lang:1985}
William Fulton and Serge Lang.
\newblock {\em Riemann-{R}och algebra}, volume 277 of {\em Grundlehren der mathematischen Wissenschaften [Fundamental Principles of Mathematical Sciences]}.
\newblock Springer-Verlag, New York, 1985.

\bibitem[Gro17]{Gross:2017}
Philipp Gross.
\newblock Tensor generators on schemes and stacks.
\newblock {\em Algebr. Geom.}, 4(4):501--522, 2017.

\bibitem[Hal22]{Hall:2022}
Jack Hall.
\newblock Further remarks on derived categories of algebraic stacks.
\newblock \href{https://arxiv.org/abs/2205.09312}{arXiv:2205.09312v4}, 2022.

\bibitem[HLP23]{Halpern-Leistner/Preygel:2023}
Daniel Halpern-Leistner and Anatoly Preygel.
\newblock Mapping stacks and categorical notions of properness.
\newblock {\em Compos. Math.}, 159(3):530--589, 2023.

\bibitem[HNR19]{Hall/Neeman/Rydh:2019}
Jack Hall, Amnon Neeman, and David Rydh.
\newblock One positive and two negative results for derived categories of algebraic stacks.
\newblock {\em J. Inst. Math. Jussieu}, 18(5):1087--1111, 2019.

\bibitem[HP24]{Hall/Priver:2024}
Jack Hall and Kyle Priver.
\newblock A generalized {B}ondal--{O}rlov full faithfulness criterion for {D}eligne--{M}umford stacks.
\newblock \href{https://arxiv.org/abs/2405.06229}{arXiv:2405.06229}, 2024.

\bibitem[HR15]{Hall/Rydh:2015}
Jack Hall and David Rydh.
\newblock Algebraic groups and compact generation of their derived categories of representations.
\newblock {\em Indiana Univ. Math. J.}, 64(6):1903--1923, 2015.

\bibitem[HR17]{Hall/Rydh:2017}
Jack Hall and David Rydh.
\newblock Perfect complexes on algebraic stacks.
\newblock {\em Compos. Math.}, 153(11):2318--2367, 2017.

\bibitem[HR18]{Hall/Rydh:2018}
Jack Hall and David Rydh.
\newblock Addendum to ``\'{E}tale d\'{e}vissage, descent and pushouts of stacks'' [{J}. {A}lgebra 331 (1) (2011) 194--223] [{MR}2774654].
\newblock {\em J. Algebra}, 498:398--412, 2018.

\bibitem[HR19]{Hall/Rydh:2019}
Jack Hall and David Rydh.
\newblock Coherent {Tannaka} duality and algebraicity of {Hom}-stacks.
\newblock {\em Algebra Number Theory}, 13(7):1633--1675, 2019.

\bibitem[HR23]{Hall/Rydh:2023}
Jack Hall and David Rydh.
\newblock Mayer-{V}ietoris squares in algebraic geometry.
\newblock {\em J. Lond. Math. Soc. (2)}, 107(5):1583--1612, 2023.

\bibitem[IT19]{Iyengar/Takahashi:2019}
Srikanth~B. Iyengar and Ryo Takahashi.
\newblock Openness of the regular locus and generators for module categories.
\newblock {\em Acta Math. Vietnam.}, 44(1):207--212, 2019.

\bibitem[Kol11]{Kollar:2011}
J{\'a}nos Koll{\'a}r.
\newblock Simultaneous normalization and algebra husks.
\newblock {\em Asian J. Math.}, 15(3):437--450, 2011.

\bibitem[Lan24]{Lank:2024}
Pat Lank.
\newblock Descent conditions for generation in derived categories.
\newblock {\em Journal of Pure and Applied Algebra}, 228(9):107671, September 2024.

\bibitem[Let20]{Letz:2020}
Janina~C. Letz.
\newblock Generation time in derived categories.
\newblock {\em ProQuest Dissertations and Theses}, 2020.

\bibitem[LN07]{Lipman/Neeman:2007}
Joseph Lipman and Amnon Neeman.
\newblock Quasi-perfect scheme-maps and boundedness of the twisted inverse image functor.
\newblock {\em Illinois J. Math.}, 51(1):209--236, 2007.

\bibitem[LO24]{Lank/Olander:2024}
Pat Lank and Noah Olander.
\newblock Approximation by perfect complexes detects rouquier dimension.
\newblock \href{https://arxiv.org/abs/2401.10146}{arXiv:2401.10146}, 2024.

\bibitem[Mat16]{Mathew:2016}
Akhil Mathew.
\newblock The {Galois} group of a stable homotopy theory.
\newblock {\em Adv. Math.}, 291:403--541, 2016.

\bibitem[Nee96]{Neeman:1996}
Amnon Neeman.
\newblock The {G}rothendieck duality theorem via {B}ousfield's techniques and {B}rown representability.
\newblock {\em J. Amer. Math. Soc.}, 9(1):205--236, 1996.

\bibitem[Nee21]{Neeman:2021}
Amnon Neeman.
\newblock Strong generators in {{\(\mathbf{D}^{\mathrm{perf}}(X)\)}} and {{\(\mathbf{D}^b_{\mathrm{coh}}(X)\)}}.
\newblock {\em Ann. Math. (2)}, 193(3):689--732, 2021.

\bibitem[Nee24]{Neeman:2022}
Amnon Neeman.
\newblock Bounded {$t$}-structures on the category of perfect complexes.
\newblock {\em Acta Math.}, 233(2):239--284, 2024.

\bibitem[Ola23]{Olander:2023}
Noah Olander.
\newblock Ample line bundles and generation time.
\newblock {\em J. Reine Angew. Math.}, 800:299--304, 2023.

\bibitem[Orl04]{Orlov:2004}
D.~O. Orlov.
\newblock Triangulated categories of singularities and {D}-branes in {Landau}-{Ginzburg} models.
\newblock In {\em Algebraic geometry. Methods, relations, and applications. Collected papers. Dedicated to the memory of Andrei Nikolaevich Tyurin.}, pages 227--248. Moscow: Maik Nauka/Interperiodica, 2004.

\bibitem[Pen24]{Peng:2024}
Fei Peng.
\newblock Equivalences of derived categories of sheaves on tame stacks.
\newblock \href{https://arxiv.org/abs/2405.19676}{arXiv:2405.19676}, 2024.

\bibitem[Rou08]{Rouquier:2008}
Rapha{\"e}l Rouquier.
\newblock Dimensions of triangulated categories.
\newblock {\em J. \(K\)-Theory}, 1(2):193--256, 2008.

\bibitem[Ryd11]{Rydh:2011}
David Rydh.
\newblock \'{E}tale d\'{e}vissage, descent and pushouts of stacks.
\newblock {\em J. Algebra}, 331:194--223, 2011.

\bibitem[Ryd15]{Rydh:2015}
David Rydh.
\newblock Noetherian approximation of algebraic spaces and stacks.
\newblock {\em J. Algebra}, 422:105--147, 2015.

\bibitem[Ryd23]{Rydh:2023}
David Rydh.
\newblock Absolute noetherian approximation of algebraic stacks.
\newblock \href{https://arxiv.org/abs/2311.09208}{arXiv:2311.09208}, 2023.

\bibitem[Sha77]{Sharp:1977}
Rodney~Y. Sharp.
\newblock The dimension of the tensor product of two field extensions.
\newblock {\em Bull. Lond. Math. Soc.}, 9:42--48, 1977.

\bibitem[Sos14]{Sosna:2014}
Pawel Sosna.
\newblock Scalar extensions of triangulated categories.
\newblock {\em Appl. Categ. Struct.}, 22(1):211--227, 2014.

\bibitem[{Sta}24]{StacksProject}
The {Stacks Project Authors}.
\newblock \textit{Stacks Project}.
\newblock \url{https://stacks.math.columbia.edu}, 2024.

\bibitem[Tot04]{Totaro:2004}
Burt Totaro.
\newblock The resolution property for schemes and stacks.
\newblock {\em J. Reine Angew. Math.}, 577:1--22, 2004.

\bibitem[Vam78]{Vamos:1978}
P.~Vamos.
\newblock On the minimal prime ideals of a tensor product of two fields.
\newblock {\em Math. Proc. Camb. Philos. Soc.}, 84:25--35, 1978.

\end{thebibliography}

\end{document}